\documentclass{amsart}
\usepackage{amssymb, amsmath, amsfonts, amsthm, mathtools, mathrsfs, xspace, tikz-cd,enumitem}
\usepackage{graphicx}

\setlist[enumerate]{
  label=(\arabic*),
  leftmargin=*,
  labelsep=0.5em
}

\theoremstyle{definition}
\newtheorem{theorem}{Theorem}[section]
\newtheorem{lemma}[theorem]{Lemma}
\newtheorem{proposition}[theorem]{Proposition}
\newtheorem{definition}[theorem]{Definition}

\newtheorem{claim}[theorem]{Claim}
\newtheorem{notation}[theorem]{Notation}
\newtheorem{remark}[theorem]{Remark}
\newtheorem{mainthm}{Theorem}
\newtheorem{maincor}[mainthm]{Corollary}

\renewcommand{\subset}{\subseteq}
\newcommand{\action}{\curvearrowright}
\newcommand{\Mat}{\mathbb{M}}
\newcommand{\C}{\mathbb{C}}
\newcommand{\Cstaralg}{$C^{\ast}$-algebra\xspace}
\newcommand{\Cstaralgs}{$C^{\ast}$-algebras\xspace}
\newcommand{\longto}{\longrightarrow}
\newcommand{\bidual}{\ast\ast}
\newcommand{\twoone}{$\mathrm{II}_1$\xspace}
\newcommand{\mintensor}{\otimes_{\mathrm{min}}}
\newcommand{\barotimes}{\overline{\otimes}}
\newcommand{\bartensor}{\overline{\otimes}}
\newcommand{\algtensor}{\otimes_{\mathrm{alg}}}

\mathtoolsset{showonlyrefs=true}

\begin{document}

\title[Relative bi-exactness for graph-wreath products]{Relative bi-exactness and structural results for graph-wreath product von Neumann algebras}
\author{Taisuke Hoshino}
\date{\today}

\begin{abstract}
We study relative bi-exactness of graph product and graph-wreath product group von Neumann algebras. 
In particular, 
we obtain the relative bi-exactness for graph product von Neumann algebras $LH_{\Gamma}=\ast_{v,\Gamma} LH_v$ and graph-wreath product von Neumann algebras $L(H_{\Gamma}\rtimes G)=(\ast_{v,\Gamma} LH)\rtimes G$, assuming that the component groups are exact. 
We adopt the $C^{\ast}$-algebraic method of Ozawa for the proof. 
As an application, for a certain class of graph-wreath products, we establish the rigidity result for the quotient graph $G\backslash\Gamma$ under stable isomorphism. 
Furthermore, we obtain a new family of prime \twoone factors. 
\end{abstract}

\maketitle

\section{Introduction and main results}

Historically, since the beginning of the study of von Neumann algebras, important examples have sometimes been constructed from simpler objects such as groups and group actions on spaces.  
At other times, they are constructed from simpler algebras, using constructions such as tensor products and amalgamated free products.
In this context, rigidity problems ask how much of the information of the building blocks is preserved in the resulting von Neumann algebra.
These problems have been a central topic in the theory of von Neumann algebra. For example,
from this point of view, Connes' celebrated work on amenable von Neumann algebras \cite{con} showes a complete absence of rigidity
for algebras arising from amenable i.c.c.\@ groups. 

At the beginning of the 21st century, two powerful methods 
to analyze the relative position of commuting von Neumann subalgebras were developed.
One is Popa's deformation/rigidity technique and its variants \cite{pop3, pop1}, 
and the other is Ozawa's notion of bi-exactness for groups \cite{oza1}.  
Bi-exactness is a property characterized by several equivalent conditions
including amenability of a certain action on their boundaries; 
typical examples include amenable groups and hyperbolic groups (see Section \ref{biexact} for detail). 
In this article, we focus on the bi-exactness approach.

These methods brought significant progress to the study of
 (amalgamated) free product von Neumann algebras \cite{ipp, oza1}.  
They provided alternative proofs of the primeness of free group factors, originally established by Ge \cite{ge} using free entropy, and moreover yielded rigidity results for tensor products and free products 
\cite{ipp, op, oza2}. 
Another class of von Neumann algebras extensively studied by these techniques is the wreath product von Neumann algebra.  
This development led to the discovery of the first example of $W^{\ast}$-superrigid groups \cite{ipv}, which are groups $G$ such that every group $H$ with isomorphic group von Neumann algebras $LG\cong LH$ satisfies $G\cong H$.

The goal of this article is to establish bi-exactness results for \emph{graph product} and \emph{graph-wreath product} groups and to apply the results to the structural problem of their group
von Neumann algebras.
Our argument is an extension of the preceding one for free products and wreath products. 
Given a graph $\Gamma =(V\Gamma, E\Gamma) $
and a family of groups $(G_v)_{v\in V\Gamma}$, 
the \emph{graph product group} $G_{\Gamma}$ is defined by
\[
G_{\Gamma} = \langle (G_v)_{v\in V\Gamma} \mid [G_v, G_w]= 1 \text{ }((v,w)\in E\Gamma) \rangle. 
\]
This is also viewed as 
an iterated amalgamated free product construction. 
It was first introduced in the group-theoretic setting by Green \cite{gre}, and later brought into the operator algebraic context by Caspers and Fima \cite{cf}. 
The rigidity problem in this setting concerns 
whether one can recover the underlying graph or the vertex groups 
from the resulting groups or von Neumann algebras. 
While this problem has been extensively studied in the group theory, 
graph product von Neumann algebras have attracted significant attention only in the last decade. Several rigidity-type results are obtained now; 
see e.g.\@ \cite{cdd1, cdd2, dv, hi} for recent progress. 
These results are based on 
the deformation/rigidity theory for amalgamated free products. 

These rigidity-type results for graph product von Neumann algebras
motivate our focus on the bi-exactness approach. 
In the operator algebraic context, 
bi-exactness is often understood as the
(AO)-type condition. 
Along this line, in \cite{cas}
Caspers established (AO)-type property for right-angled Hecke algebras, 
which are special cases of graph product von Neumann algebras. 
Then, recently in \cite[Section 6.2]{bcc}, Borst, Caspers, and Chen obtained 
a solidity type result for graph products by assuming and making use of
(AO)-type property for each vertex von Neumann algebra. 
And in another direction, by a geometric approach, 
bi-exactness-type results for graph products 
were obtained by Oyakawa in \cite[Theorem 4.9, 5.7]{oya2}, under certain technical assumptions.

Our first main result establishes a bi-exactness-type result for graph products, removing the technical constraints of the preceding result \cite[Theorem 4.9]{oya2}. 

\begin{mainthm}
Let $\Gamma=(V\Gamma, E\Gamma)$ be a graph and $(H_v)_{v\in V\Gamma}$ a family of exact groups.  Then the graph product $H_{\Gamma}$ is bi-exact relative to $\{ \langle H_w \mid v=w \text{ or }(v,w)\in E\Gamma \rangle\}_{v\in V\Gamma}$.  

Moreover, if $(H_v)$ are bi-exact for all $v\in V\Gamma$, then $H_{\Gamma}$ is bi-exact relative to $\{ \langle H_w \mid (v,w)\in E\Gamma \rangle\}_{v\in V\Gamma}$.  
\end{mainthm}

The proof of Theorem A relies on the $C^{\ast}$-algebraic 
technique developed in \cite{oza2}. 
In addition, the sufficient condition for the nuclearity of 
reduced amalgamated free products, obtained
in \cite{has},
also play a part. 


We next consider
graph-wreath product von Neumann algebras. 
Given a group $H$ and a group action on a graph $G\action \Gamma$, 
denote by $H_{\Gamma}$ the graph product with the underlying graph $\Gamma$
and the vertex groups $(H_v)_{v\in V\Gamma}$ all isomorphic to $H$. Then, there is a Bernoulli-type action
$\sigma : G \to \mathrm{Aut}(H_{\Gamma})$ such that $\sigma_g$ sends $h\in H_v$
to $h\in H_{gv}$. 
The semi-direct product $H_{\Gamma}\rtimes_{\sigma} G$ is called \emph{graph-wreath product group}. 
Unlike graph products and wreath products, 
this construction has received comparatively little attention. 
Nevertheless, in the context of bi-exactness, 
Oyakawa \cite[Theorem 5.7]{oya2} 
recently proved, using a geometric approach, that a certain class of graph-wreath products satisfies a relative version of bi-exactness. 
We strengthen this result by adopting an analytic method, along the lines of the approach developed in \cite{hik}. 

\begin{mainthm}
Let $H$ be an exact group, $\Gamma=(V\Gamma, E\Gamma)$ a graph, and $G\action \Gamma$ an action of a bi-exact group $G$ on $\Gamma$. If the isotropy groups of $G\curvearrowright V\Gamma$ are all finite and $G\action V\Gamma$ has finitely many orbits, then the graph-wreath product $H_{\Gamma}\rtimes G$ is bi-exact relative to $\{ \langle H_w \mid v=w \text{ or }(v,w)\in E\Gamma \rangle\}_{v\in V\Gamma}$. 

Moreover, if $H$ is bi-exact, then $H_{\Gamma}\rtimes G$ is bi-exact relative to $\{ \langle H_w \mid (v,w)\in E\Gamma \rangle\}_{v\in V\Gamma}$. 
\end{mainthm}

\begin{remark}

(1) The relative bi-exactness proven in Theorems A and B are in the strongest form possible, given that $H$ is a general 
nonamenable bi-exact group. Indeed, if $H$ is non-amenable, then for any $v\in V\Gamma$, the commutant of $W^{\ast}\{LH_w\mid (v,w)\in E\Gamma\}\subset LH_{\Gamma}$ contains the nonamenable subalgebra $LH_v$.  

(2) Compared to the results in \cite{oya2}, Theorems A and B are stronger than Theorem 4.9 and 5.7 (2) in \cite{oya2} respectively, but Theorem A cannot reproduce \cite[Theorem 5.7 (1)]{oya2}, where the genuine bi-exactness is proven under the assumption of $H$ being finite and $\Gamma$ admitting no small circuits. 
\end{remark}

Now we turn to our rigidity results of graph-wreath product von Neumann algebras. 
In the case of graph products, it is now known that one can recover 
the underlying graph from the resulting graph product von Neumann algebras
under some technical assumption on the graphs and vertex algebras  \cite{bcc, hi}. 
Our motivation is to generalize this result to graph-wreath products. 

To describe the result we recall the definition of 
the quotient graph of an action on a graph $G\action \Gamma$;
The quotient graph $G\backslash \Gamma$ consists of the set of vertices $G\backslash V\Gamma$ and the the set of edges $G\backslash E\Gamma$, 
where $G(v,w)\in G\backslash E\Gamma$ is an edge connecting $Gv$ and $Gw$. 
Beware that these graphs in general admit loops and multiple edges. 

We finally state our rigidity result for graph-wreath product von Neumann algebras. 
A simple graph $\Gamma$ is called \emph{untransvectable} 
if, for each pair of vertices $v\neq w\in V\Gamma$, there is a vertex $u\neq w$
which is adjacent to $v$ but not adjacent to $w$. 
The \emph{girth} of $\Gamma$ is the length of the shortest circuit in it. The girth of a tree is defined to be $\infty$. 
For the precise definition see Section \ref{chapgraph}. 

\begin{mainthm}
Let $H_1$ and $H_2$ be infinite exact groups, 
$\Gamma_1=(V\Gamma_1, E\Gamma_1)$ and $\Gamma_2=(V\Gamma_2, E\Gamma_2)$ infinite connected untransvectable locally finite graphs with their girth at least 5, and $G_1\curvearrowright \Gamma_1$ and $G_2\action \Gamma_2$ actions of an infinite bi-exact group $G_1$, $G_2$, respectively. 

Assume that $H_i$ is i.c.c.\@ and $G_i\action V\Gamma_i$ is a free action with finitely many orbits for $i=1,2$. 

If the graph-wreath product $\mathrm{II}_1$ factors 
$L((H_1)_{\Gamma_1}\rtimes G_1)$ and $L((H_2)_{\Gamma_2}\rtimes G_2)$ are stably isomorphic, 
then $|G_1\backslash V\Gamma_1|=|G_2\backslash V\Gamma_2|$. 

Moreover, if $H_1$ and $H_2$ have Kazhdan's Property (T), then the quotient graphs $G_1\backslash \Gamma_1$ and $G_2\backslash \Gamma_2$ are isomorphic. 
\end{mainthm}

The proof of Theorem C relies on Theorem B and
an argument similar to the one in \cite[Section 6.1]{hi}. 

This immediately yields the following corollary. 

\begin{maincor}
Let $\mathbb{F}_n$ be the free group of rank $n$, and $\mathbb{F}_n\action T_n$ its left regular action on the Cayley graph of $\mathbb{F}_n$ with the canonical generators. Let $m, n\geqq 2$ be positive integers, and $H$ and $K$ infinite exact groups with Property (T). 
If the graph-wreath product $\mathrm{II}_1$ factors 
$L(H_{T_m}\rtimes \mathbb{F}_m)$ and $L(K_{T_n}\rtimes \mathbb{F}_n)$
are stably isomorphic, then $m=n$. 
\end{maincor}

Note that, though in Theorem C we obtained the rigidity result for
the quotient graph $G_i\backslash \Gamma_i$, the author does not know if any further rigidity results 
for the groups $G_i$ or the graphs $\Gamma_i$ are possible. 

Finally, as a byproduct, 
we obtain the following primeness result for
graph-wreath products, by combining Theorem B with the method in \cite[Section 7]{im}. 
Here, a \twoone factor $M$ is said to be
\emph{prime} if it cannot be written as the tensor product 
$M=P\overline{\otimes} Q$ of
two \twoone factors $P$, $Q$. 

\begin{mainthm}
Let $H$ be an infinite exact group, $\Gamma=(V\Gamma, E\Gamma)$ an infinite and locally finite graph, and $G\curvearrowright \Gamma$ an action of an infinite bi-exact group $G$ on $\Gamma$. If the isotropy groups of $G\curvearrowright V\Gamma$ are all finite and $G\action V\Gamma$ has finitely many orbits, then the graph-wreath product von Neumann algebra $L(H_{\Gamma}\rtimes G)$ is prime.
\end{mainthm}


\subsection*{Acknowledgement}
The author would like to thank his supervisor Yasuyuki Kawahigashi for many valuable comments. 
The author also would like to thank Yusuke Isono 
for several fruitful discussions and especially 
for drawing his attention to \cite{hik}. This is a master thesis for the University of Tokyo and the author is supported by WINGS-FMSP program.

\section{Preliminaries}

Thorughout this paper, we assume that all groups are countable and discrete, and all graphs are (not necessarily finite) countable simple graphs except in the statement of Theorem C. 

For the basis of completely positive maps and tracial von Neumann algebras we refer to \cite{bo}
and \cite{ap} respectively. 

We use the following notation. 

For a tracial von Neumann algebra $(M, \tau)$, let $L^2M$ be its standard representation, and denote by $\| \cdot \|_2$ its norm 
induced by $\tau$; $\| x \|_2 = \tau(x^{\ast}x)^{\frac{1}{2}}$. 
We denote by $J : L^2M \to L^2M$ the antilinear isometry extending 
$M\ni x \mapsto x^{\ast}\in M$.  

We denote algebraic tensor products, minimal tensor products, and von Neumann algebraic tensor products by $\algtensor$, $\mintensor$, and $\bartensor$ respectively. 

For a von Neumann algebra $M$ and its subalgebra $B$, we denote by 
\begin{align}
\mathcal{N}_M(B) &= \{ u\in \mathcal{U}(M) \mid uBu^{\ast}=B \}, \\
\mathcal{QN}_M(B) &= \{ x\in M \mid \text{there are } x_1, \ldots x_n\in M \text{ such that }
xB\subset \sum_{i=1}^n Bx_i, Bx\subset \sum_{i=1}^nx_iB\}
\end{align}
the set of normalizers and quasi-normalizers of $B$ inside $M$,
respectively. 

\subsection{Intertwining-by-bimodule technique}

The following theorem of Popa, which deals with the position of two von Neumann subalgebras, has been an essential tool in the field for more than a decade. 

\begin{theorem}[\cite{pop3, pop1}]
Let $M$ be a tracial von Neumann algebra, $p_0\in M$ a projection, and let $P\subset p_0Mp_0$, $Q\subset M$ be its von Neumann subalgebras. 
Then the following are equivalent. 
\begin{enumerate}
\item There does not exist a net of unitaries $(u_i)_i\in \mathcal{U}(P)$ such that $\| E_Q(au_ib) \|_2 \to 0$ holds for every $a,b\in M$. 
\item There exists a $P$-$Q$-subbimodule of $p_0L^2M$ which is finitely generated as a right $Q$-module. 
\item There exist a positive integer $n$, a normal $\ast$-homomorphism $\pi : P\to \Mat_n(Q)$, and a partial isometry $v\in \Mat_{1,n}(M)$ such that 
$v\pi(x)=xv$ for every $x\in P$. 
\item There exist nonzero projections $p\in P$, $q\in Q$, a normal $\ast$-homomorphism $\theta:pPp\to qQq$, and a partial isometry $v\in pMq$ such that $v\theta(x)=xv$ for every $x\in pPp$.
\end{enumerate}

We denote $P\prec_M Q$ when one (or, equivalently, all) of the above holds, and $P\nprec_M Q$ otherwise. 
We often drop the subscript $M$ when the algebra in question is obvious. 

Moreover, even if the subalgebra $Q\subset q_0Mq_0$ does not share the unit with $M$, 
the equivalence of (1) and (4) still holds. We write in this case $P\prec_M Q$ as well. 

\end{theorem}
In the situation above, we write $P\prec_M^s Q$ when $Pp'\prec_M Q$ holds for every nonzero projection $p'\in \mathcal{Z}(P'\cap pMp)$. 

We record several lemmas for later use. 

\begin{lemma}[\textup{\cite[Lemma 3.7]{vaes}}]
Let $M$ be a tracial von Neumann algebra and $P, Q, R\subset M$ its von Neumann subalgebras. If $P\prec_M Q$ and $Q\prec_M^s R$, then $P\prec_M R$. 
\end{lemma}

\begin{lemma}[\textup{\cite[Remark 3.8]{vaes}}]
Let $M$ be a tracial von Neumann algebra, $P\subset pMp$ a von neumann subalgebra, and $Q, Q_1, \ldots, Q_n\subset M$ von Neumann subalgebras. If $P\prec_M Q$ and $P\nprec_M Q_i$ for each $i\in \{ 1,2,\ldots, n\}$, then 
there exist a positive integer $n$, $\ast$-homomorphism $\pi : P\to \Mat_n(Q)$ and a partial isometry $v\in \Mat_{1,n}(M)$ such that 
$v\pi(x)=xv$ for every $x\in P$ and $\pi(pPp)\nprec_{\Mat_n(M)} Q_i$ for each $i$, as in Theorem 2.1 (3). 
Likewise, there exist nonzero projections $p\in P$, $q\in Q$, a normal $\ast$-homomorphism $\theta:pPp\to qQq$, and a partial isometry $v\in pMq$ such that $v\theta(x)=xv$ for every $x\in pPp$
and $\theta(pPp)\nprec_M Q_i$ for each $i$. 
\end{lemma}

\begin{lemma}[\textup{\cite[Remark 3.3]{vaes}}]
Let $M$ be a tracial von Neumann algebra, $P\subset pMp$ a von Neumann subalgebra, and $Q_1, \cdots, Q_n, \cdots \subset M$ a countable number of von Neumann subalgebras. If $P\nprec_M Q_i$ holds for each $i\in \mathbb{N}$, then there is a net of unitaries $(u_j)_j\in \mathcal{U}(P)$ such that $\| E_{Q_i}(au_jb) \|_2 \to 0$ holds for every $i$ and $a,b\in M$. 
\end{lemma}

\subsection{Relative bi-exactness}\label{biexact}

Let $G$ be a group, and $\mathcal{G}$ a family of \textbf{subsets} of $G$. 
First we fix some terminology.
\begin{itemize} 

\item A subset $T$ of $G$ is called \emph{small relative to $\mathcal{G}$} if there is a finite number of elements $a_1, a_2, \ldots, a_m, b_1, b_2, \ldots, b_m\in G$ and $S_1, S_2, \ldots , S_m\in \mathcal{G}$ such that $T\subset \bigcup_i a_iS_ib_i$. 

\item A complex-valued function $f\in \ell^{\infty}(G)$ 
satisfies $f(g)\to 0 \text{ } (g\to \infty / \mathcal{G})$ when, 
for every $\varepsilon > 0$, the set 
$\{ g\in G \mid |f(g)|>\varepsilon \}$
is small relative to $\mathcal{G}$. 

The ideal of $\ell^{\infty} G$, consisting of all functions $f$ which satisfy $f\in \ell^{\infty}(G)$ 
is denoted $c_0(G, \mathcal{G})$. 

\item We define a $C^{\ast}$-algebra $\mathcal{K}(\mathcal{G})$ by 
$\mathcal{K}(\mathcal{G}):=\overline{c_0(G, \mathcal{G})B(\ell^2(G))c_0(G, \mathcal{G})}^{\|\cdot\|}$. 
Note that the multiplier algebra $\mathcal{M}(\mathcal{K}(G))\subset \mathcal{B}(\ell^2 G)$ of $\mathcal{K}(\mathcal{G})$ contains $\ell^{\infty} G$, $C^{\ast}_{\lambda}(G)$, and $C^{\ast}_{\rho}(G)$. 
\end{itemize}

\begin{definition}
We say that the group $G$ is \emph{bi-exact relative to $\mathcal{G}$} if there is a map $\mu : G\to \mathrm{Prob}(G)$ such that 
\[
\| \mu (gsh) - g.\mu(s) \|_1 \to 0 \text{ }(s\to \infty / \mathcal{G})
\]
holds for every $g,h\in G$. 
\end{definition}

This property has been related to the (AO)-type property and the concept of amenable actions. 

\begin{proposition}[\textup{\cite[Proposition 15.1.4, 15.2.3]{bo}}]
Let $G$ and $\mathcal{G}$ be as above. Then the following
conditions are equivalent.  
\begin{itemize}
\item The group $G$ is bi-exact relative to $\mathcal{G}$. 
\item There exists a unital completely positive (u.c.p.\@ for short) map 
\[
\theta : C^{\ast}_{\lambda}(G) \mintensor C^{\ast}_{\rho}(G) 
\to B(\ell^2(G))
\]
such that $\theta(a\otimes b)-ab \in \mathcal{K}(\mathcal{G})$. 
\item The action 
\[
G\curvearrowright S_{\mathcal{G}}(G):=\{ f\in \ell^{\infty}(G)\mid f(\bullet)-f(\bullet \cdotp t)\in c_0(G, \mathcal{G}) (\forall t\in G)\}
\] is amenable. 
\item The left-right regular action of $G\times G$ on $\ell^{\infty}(G)/c_0(G,\mathcal{G})$ is amenable. 
\end{itemize} 
\end{proposition}

When $\mathcal{G}$ consists only of finite subsets, 
we write $S_{\mathcal{G}}(G)$ by $S(G)$, 
and we simply call $G$ bi-exact if one of the conditions above holds. 

For two families of subgroups $\mathcal{G}$ and $\mathcal{G}'$, 
let $\mathcal{G}\wedge \mathcal{G}'$ be the family of the subsets of the form $F\cap F'$, 
where $F$, $F'$ are small relative to $\mathcal{G}$, $\mathcal{G}'$, respectively.  
Note that, by definition, 
$\mathcal{K}(\mathcal{G}\wedge\mathcal{G}')=\mathcal{K}(\mathcal{G})\cap\mathcal{K}(\mathcal{G}')$. 

We record the following solidity type result:

\begin{theorem}[\cite{bo}, Theorem 15.1.5]
Let $G$ be an infinite i.c.c. discrete group and 
$\mathcal{G}$ a family of subgroups of $G$. 
Suppose that $G$ is bi-exact relative to $\mathcal{G}$. 
Denote $M=LG$. 

Let $n$ be a positive integer and $p\in\Mat_n (M)=\Mat_n(\C)\otimes M$ a projection. 
Then, for any von Neumann subalgebra $P$ of $p\Mat_n(M)p$ satisfying
$P\nprec_{\Mat_n(M)} 1\otimes LH$ for every $H\in \mathcal{G}$, the relative commutant 
$P'\cap L\Gamma$ is amenable. 
\end{theorem}

\subsection{Graphs and graph products for groups}\label{chapgraph}

First we fix the notation for graphs. 

\begin{notation}
A \emph{(simple) graph} $\Gamma=(V\Gamma, E\Gamma)$ consists of a set of vertices $V\Gamma$ and a set of edges $E\Gamma\subset V\Gamma\times V\Gamma$ which satisfies $(v,v)\notin E\Gamma$ and $(v,w)\in E\Gamma \Longrightarrow (w,v)\in E\Gamma$ for any $v,w\in V\Gamma$. We call two vertices $v,w\in V\Gamma$ are \emph{adjacent} when $(v,w)\in E\Gamma$. 

For a graph $\Gamma$ and $v\in V\Gamma$, we write
\begin{align}
\mathrm{Lk}_{\Gamma}v &= \{ w\in V\Gamma \mid (v,w)\in E\Gamma \}, \\ 
\mathrm{St}_{\Gamma}v &= \mathrm{Lk}_{\Gamma}v\cup\{ v \}. 
\end{align}
and call $|\mathrm{Lk}_{\Gamma}v|$ the \emph{degree} of $v$. 
A graph is said to be \emph{locally finite} if the degree of each vertex is finite. 

We denote $\mathrm{Lk}_{\Gamma}E=\cap_{v\in E} \mathrm{Lk}_{\Gamma}v$ for each subset $E\subset V\Gamma$. 
We often drop the subscript $\Gamma$ if the graph in question is obvious. 
 
A \emph{path} in $\Gamma$ is a finite sequence of vertices in $\Gamma$ whose adjacent elements are also adjacent in $\Gamma$, and define the length of a path to be the \emph{length} of the sequence minus one.
$\Gamma$ is said to be \emph{connected} if for every $v,w\in V\Gamma$ there is a path whose initial vertex is $v$ and whose final vertex is $w$. 
 
A \emph{circuit} in $\Gamma$ is a path of length at least 3 in which all vertices are distinct except that the initial and final vertices coincide. We write the \emph{girth} of $\Gamma$, the length of the shortest circuit in $\Gamma$, by $\mathrm{girth}\Gamma$. We set $\mathrm{girth}\Gamma = \infty$ if no circuits exist in $\Gamma$. 

A graph $\Gamma$ is called \emph{untransvectable}
if $\mathrm{Lk}_{\Gamma}v \nsubseteq \mathrm{St}_{\Gamma}w $ holds for each pair of vertices $v\neq w\in V\Gamma$. 
It is easily seen that any connected untransvectable graphs with at least 3 vertices have no vertices of degree 0 or 1. 

\end{notation}

Next we define the graph product of groups. 

\begin{definition}
Let $\Gamma = (V\Gamma, E\Gamma)$ be a graph and $(G_v)_{v\in V\Gamma}$ a family of nontrivial groups indexed by $V\Gamma$. Then the \emph{graph product} $G_{\Gamma}=\ast_{v,\Gamma} G_v$ of $(G_v)_{v\in V\Gamma}$ is defined by
\[
G_{\Gamma} := \langle (G_v)_{v\in V\Gamma} \mid [G_v, G_w]= 1 (\forall (v,w)\in E\Gamma) \rangle. 
\] 
\end{definition}

For $E\subset V\Gamma$ we often write by $G_E$ the subgroup of $G_{\Gamma}$ generated by $\{ G_v\mid v\in E\}$. 

\begin{definition}
Let $\Gamma$, $(G_v)_{v\in V\Gamma}$, and $G_{\Gamma}$ be as in the definition above.
\begin{itemize}
\item Let $g_i\in G_{v_i}\backslash  \{e_{G_{v_i}}\}$ for $1\leqq i\leqq n$. A presentation $g=g_1g_2\cdots g_n\in G_{\Gamma}$ is called a \emph{normal form} when $n$ is the smallest among all possible such presentations of $g$. The smallest possible $n$ is denoted $\| g\|$. 
\item A sequence of vertices $v_1, v_2, \ldots, v_n\in V\Gamma$ is said to be \emph{irreducible} when for any pair $(i,j)$ with $1\leqq i<j\leqq n$ and $v_i=v_j$ there exists $k$ with $i<k<j$ such that $(v_i, v_k)\notin E\Gamma$. 
\end{itemize}
\end{definition}

The following theorem of Green \cite{gre} is useful when dealing with presentations of elements in a graph product. 

\begin{theorem}[Green's theorem, see also \textup{\cite[Theorem 2.14]{oya1}}]

Let $\Gamma$, $(G_v)_{v\in V\Gamma}$, and $G_{\Gamma}$ be as in Definition 2.9. Then, $g=g_1g_2\cdots g_n\neq e\in G_{\Gamma}$, a presentation of an element $g$ in $G_{\Gamma}$ with $g_i\in G_{v_i}\backslash \{ e \}$, is a normal form if and only if the sequence $v_1, v_2, \ldots, v_n\in V\Gamma$ is irreducible. Moreover, for any two normal forms of $g$, one can rearrange one of them into the other. 
\end{theorem}

In particular, for every $g\in G_{\Gamma}$, the set $\{ v_1, \ldots , v_n\}\subset V\Gamma$ is independent of the choice of the normal form $g=g_1g_2\cdots g_n\in G_{\Gamma}, (g_i\in G_{v_i}, \text{ } v_1, v_2, \ldots, v_n\in V\Gamma)$. We denote this set by $\mathop{supp} g\subset V\Gamma$. 

We will make use of the proof of the following easy lemma later on. 

\begin{lemma}
Let $\Gamma$, $(G_v)_{v\in V\Gamma}$, and $G_{\Gamma}$ be as above. 
For $v\in V\Gamma$, $g\in G_v$, and $h\in G_{\Gamma}$, we have 
\[
\|h \|-1 \leqq \| gh \|\leqq \|h \|+1. 
\]
\end{lemma}
\begin{proof}
The second inequality is trivial. 
To prove the first one let $h=h_1h_2\cdots h_n$ be a normal form of $h$. 
If $\|gh \| < \| h \|$, then by Theorem 2.11 there exists an index $i$ such that $v_i=v$, $gh_i=e$ and $v_1, \ldots v_{i-1}\in \mathrm{Lk}_{\Gamma}v$. 
We claim that in this case $gh=h_1h_2\cdots h_{i-1}h_{i+1}\cdots h_n$ is a normal form of $gh$. 
In fact, when $i=1$ or $i=n$ it is trivial. Otherwise, if two syllables in the presentation were to be reducted, then one can without loss of generality assume that the two are $h_{i-1}$ and $h_{i+1}$. Then it follows $v_{i+1}=v_{i-1}\in \mathrm{Lk}_{\Gamma}v$ and therefore $[h_i, h_{i+1}]=e$. But again by Theorem 2.11 this contradicts with
the assumption that $h_1\cdots h_{i-1}h_ih_{i+1}\cdots h_n$ is a normal form. 
\end{proof}

\subsection{Graph products for operator algebras}

Fix a graph $\Gamma$. We consider a family of $C^{\ast}$-algebras with states 
$(A_v, \varphi_v)_{v\in V\Gamma}$ indexed by the set of vertices $V\Gamma$. 

We define the (reduced) graph product of $(A_v, \varphi_v)_v$ in accodance with \cite{cf}. 
Let $(\pi_v, H_v, \xi_v)$ be the GNS representation of $(A_v, \varphi_v)$. 

First let $H$ be the graph product Hilbert space of $(H_v)_v$. i.e.
\[
H = \mathbb{C}\Omega \oplus \bigoplus_{v_1v_2\cdots v_n : \text{irreducible}} H_{v_1}^o \otimes H_{v_2}^o \otimes \cdots \otimes H_{v_n}^o, 
\]
where $H_v^o = H_v\ominus \mathbb{C}\xi_v$. 

Next we define canonical representations $\lambda_v : \mathcal{B}(H_v) \to \mathcal{B}(H)$ for each $v\in V\Gamma$. Let $H(v)$ be the Hilbert space 
\[
H(v)=\mathbb{C}\Omega \oplus \bigoplus_{v_1v_2\cdots v_n, vv_1v_2\cdots v_n : \text{irreducible}} H_{v_1}^o \otimes H_{v_2}^o \otimes \cdots \otimes H_{v_n}^o. 
\]
Then, one can define a unitary map $U_v : H_v\otimes H(v) \to H$ such that
$U_v$ maps
\begin{align}
\xi_v\otimes \Omega &\longmapsto \Omega, \\
\xi\otimes \Omega &\longmapsto \xi, \\
\xi_v\otimes \eta &\longmapsto \eta, \\
\xi\otimes \eta &\longmapsto \xi\otimes \eta
\end{align}
for any $\xi\in H_v\ominus \C \xi_v$ and $\eta\in H(v)\ominus \C \Omega$. 
Let $\lambda_v = \mathrm{Ad}(U_v) \circ (\pi_v\otimes \mathrm{id})$. 
Note that $[\lambda_v(\mathcal{B}(H_v)), \lambda_w(\mathcal{B}(H_w))]=0$ whenever $(v,w)\in E\Gamma$. 

Finally we define the reduced graph product $C^{\ast}$-algebra
$A_{\Gamma}\subset \mathcal{B}(H)$ to be the one generated by
$(\lambda_v(\pi_v(A_v)))_v$. The \Cstaralg $A_{\Gamma}$ has the canonical state $\varphi$, namely the vector state of $\Omega \in H$. For a subset $F\subset V\Gamma$, we denote by $A_F$ the $C^{\ast}$-subalgebra of $A_{\Gamma}$ generated by $\{ A_v\mid v\in F\}$.  

By its definition, the graph product can be seen as a chain of the reduced amalgamated free products of the form
\[
(A_{\Gamma}, E)=(A_{\mathrm{St}v}, E) \ast_{r, A_{\mathrm{Lk}v}} (A_{\Gamma\backslash \{v \}}, E)
\]
where the maps denoted by $E$ are the suitable conditional expectations onto $A_{\mathrm{Lk}v}$. 
In particular, one can see that for every subset $F\subset V\Gamma$, 
there is a conditional expectation $E:A_{\Gamma}\to A_{F}$ which preserves the canonical state $\varphi$. 

For the case of von Neumann algebras, 
let $(B_v, \tau_v)_v$ be a family of tracial von Neumann algebras. 
Then, under the similar notation as above, the graph product (tracial) von Neumann algebra $(B_{\Gamma}\subset \mathcal{B}(H), \tau)$ is defined to be the one generated by $(\lambda_v(B_v))_v$. 

For a subset $F\subset V\Gamma$, we similarly denote by $B_F$ the von Neumann subalgebra of $B_{\Gamma}$ generated by $\{ B_v\mid v\in F\}$. 
The following lemma about the position of these algebras inside the graph product von Neumann algebra is useful. 

\begin{lemma}[\cite{bcc,dv,hi}]
Let $B_{\Gamma}$ be the graph product von Neumann algebra as above, $v\in V\Gamma$ a vertex, and $E,F\subset V\Gamma$ finite subsets. 
If each vertex algebra $B_v$ is diffuse, then the following holds;
\begin{enumerate}
\item $B_E'\cap B_{\Gamma}=\mathcal{Z}(B_E)\overline{\otimes} B_{\mathrm{Lk}E}$. 
\item $B_{\Gamma}$ is a factor if and only if so are $B_v$ for every $v\in V\Gamma$ with $\mathrm{St}_{\Gamma}v=V\Gamma$.   
\item Any right $B_E$-finite $B_E$-$B_E$-subbimodule of $L^2B_{\Gamma}$ is contained in $L^2B_{E\cup \mathrm{Lk}E}$. 
In particular, $\mathcal{QN}_{B_{\Gamma}}(B_E)''=B_{E\cup \mathrm{Lk}E}$.  
\item $B_E\prec B_F$ if and only if $E\subset F$.  
\end{enumerate}
Moreover, the following variants of (3) and (4) hold; 
\begin{enumerate}
\item[(3)'] For any projection $p\in B_E$ and von Neumann subalgebra $P\subset pB_{E}p$ which satisfies
$P\nprec_{B_E} B_{E'}$ for any proper subsets $E'\subsetneq E$, 
if an element $x\in pB_{\Gamma}$ satisfies 
$Px\subset \sum_{i=1}^k x_i B_{E}$ for some $x_1, \ldots , x_k\in B_{\Gamma}$, 
then $x\in B_{E\cup \mathrm{Lk}E}$. 
In particular, $\mathcal{QN}_{pB_{\Gamma}p}(P)''\subset B_{E\cup\mathrm{Lk}E}$. 
\item[(4)'] For any von Neumann subalgebra $P\subset pB_{\Gamma}p$ which satisfies $P\prec^s_{B_{\Gamma}} B_E$ and $P\nprec_{B_{\Gamma}} B_{E'}$ for any proper subset $E'\subsetneq E$, $P\prec_{B_{\Gamma}} B_F$ if and only if $E\subset F$.  
\end{enumerate}
\end{lemma} 

\begin{proof}
(1): see \cite[Lemma 2.16]{hi}. 

(2),(3),(4): see \cite[Lemma 6.2]{dv}. 

(3)' and (4)' follows from the proof of 
\cite[Proposition 5.8]{bcc} and \cite[Claim 3.6]{hi} respectively. 
\end{proof}

In particular (4)' of this lemma is of frequent use in Section 5. 

\subsection{Graph-wreath products}

Let $H$ be a group, $\Gamma$ a graph, and
$G\curvearrowright \Gamma$ an action of a group $G$ on $\Gamma$.  
Then, there is a Bernoulli-type action 
$\sigma : G\to \mathrm{Aut}(H_{\Gamma})$
given by the following;
\[
\sigma_g : H_v \ni h \longmapsto h \in H_{gv}
\]

Note that this action $\sigma$ extends to a trace-preserving action 
on its group von Neumann algebra $G\curvearrowright L(H_{\Gamma})$. 
This action will be denoted $\sigma$ as well in the sequel. 

From this action one can construct the crossed product group 
$H_{\Gamma}\rtimes_{\sigma} G$. We will simply write this by $H_{\Gamma}\rtimes G$.  

\begin{lemma}
In the above setting, let $M=L(H_{\Gamma}\rtimes G)$ and $B_E=L(H_E)$ for $E\subset V\Gamma$. If all the isotropy groups of the action $G\curvearrowright V\Gamma$
is finite, then the following holds. 
\begin{enumerate}
\item The trace-preserving action 
$\sigma: G\curvearrowright B_{\Gamma}$
is mixing. 

\item The von Neumann algebra $M$ is a factor. 

\item If $|H|\geqq 3$, $M$ is amenable if and only if $G$, $H$ are amenable and $\Gamma$ is complete.

\item Assume moreover that $H$ is infinite. If $E\subset V\Gamma$ is a finite subset, then every right $B_E$-finite $B_E$-$B_E$-subbimodule of $L^2M$ is contained in $L^2(B_{E\cup \mathrm{Lk}E}\rtimes G^E)$. Here we write by $G^E=\{ g\in G \mid gE=E\}$ the isotropy group of $E$.
In particular, $\mathcal{QN}_M(B_E)\subset B_{E\cup \mathrm{Lk}E}\rtimes G^E$. 

\item If $E\subset V\Gamma$ is a finite subset, then $B_E'\cap M =\mathcal{Z}(B_E) \overline{\otimes}(B_{\mathrm{Lk}E}\rtimes G_0^E)$ and therefore $\mathcal{Z}(B_E'\cap M) = \mathcal{Z}(B_E) \overline{\otimes}\mathcal{Z}(B_{\mathrm{Lk}E}\rtimes G_0^E) $, where $G_0^E=\{ g\in G\mid gv=v \text { for every }v\in E\}$. 

\end{enumerate}
\end{lemma}
\begin{proof}
(1) It suffices to prove that, for every $h, k\in H_{\Gamma}$, $\tau(u_hu_{\sigma_g(k)})=0$ holds for all but finitely many $g\in G$. But this holds unless $\mathop{supp}k\cap g \mathop{supp}k \neq \emptyset$,  
and the number of $g\in G$ satisfying this condition is finite by the assumption on $G\curvearrowright V\Gamma$. 

(2) follows directly from (1). 

(3) The first half of the only if part is trivial since $L(H_{\Gamma}\rtimes G)$ contains $LG$ and $LH$. 
For the latter half just notice that $H\ast H$ is nonamenable. 
The if part is trivial since amenability is preserved under taking direct products and crossed products. 

(4) Let $\mathcal{K}$ be a right $B_E$-finite $B_E$-$B_E$-subbimodule of $L^2M$. Take an element $\xi\in \mathcal{K}$ and $\xi_1,\xi_2, \ldots, \xi_n \in \mathcal{K}$ such that $B_E\xi\subset \sum_i \xi_iB_E$, 
and let $\xi = \sum_{g\in G} \xi^g u_g$ and $\xi_i=\sum_{g\in G} \xi_i^g u_g$ $(\xi^g, \xi_i^g\in B_{\Gamma})$
be their Fourier expansions. Then, by comparing coefficients of $u_g$, one obtains that, for every $g\in G$, 
$B_E\xi^g\subset \sum_i \xi_i^gB_{gE}$. 
Then, by Lemma 2.13 (3) and (4), if $\xi^g\neq 0$, then $E\subset gE$ and $E=gE$, since $E$ is finite, and
$\xi^g\in L^2B_{E\cup \mathrm{Lk}E}$. This verifies the claim. 

(5) By (4), it follows that $B'_E\cap M\subset B_{E\cup \mathrm{Lk}E}\rtimes G^E$. Take $x=\sum_{g\in G^E} x^gu_g\in B_E'\cap M$. Then, since $x$ commutes with $B_E$, $x^g$ satisfies $yx^g=x^g\sigma_g(y)$ for every $y\in B_E$ and $g\in G$. If $x^g\neq 0$, this means that $B_{v}x^g\subset x^g B_{gv}$ holds for every $v\in E$, and therefore $g\in G_0^E$ and $x^g\in B_E'\cap B_{E\cup \mathrm{Lk}E}=\mathcal{Z}(B_E)\overline{\otimes}B_{\mathrm{Lk}E}$ by Lemma 2.13 (1) and (4). 
\end{proof}

\subsection{Property (T)}

Kazhdan's Property (T) plays a significant role in deformation-rigidity theory.  
The facts about Property (T) we will use in this paper are the following theorem
and that the tensor product of two von Neumann algebras with Property (T) again has the property. 

\begin{theorem}[\textup{\cite[Theorem 4.3]{ipp}}, see also \textup{\cite[Theorem 5.6]{hou}}]
Let $M_1, \ldots, M_n$ be tracial von Neumann algebras and $B\subset M_i$ 
a common von Neumann subalgebra of them. 

Let $M=p\Mat_m(M_1\ast_B M_2 \ast_B\cdots \ast_B M_n)p$ be 
an amplification of the amalgamated free product. 
If a subfactor $Q\subset M$ has Property (T), 
then there is an index $i\in \{ 1,2,\ldots, n\}$ such that $Q\prec_M M_i$. 
\end{theorem}

\section{Bi-exactness for graph product von Neumann algebras}

In this section, we provide the proof of Theorem A. 
The proof makes use of the preceding results concerning 
nuclearity and extension of u.c.p.\@ maps for 
reduced amalgamated free product \Cstaralgs
\cite{oza2, has}.

\begin{proposition}
For any graph $\Gamma$, positive integers $(n_v)_{v\in V\Gamma}$ and pure states $(\phi_v)_{v\in V\Gamma}$ on the matrix algebras $(\Mat_{n_v}(\mathbb{C}))_{v\in V\Gamma}$, the reduced graph product \Cstaralg $M_\Gamma$ of $(\Mat_{n_v}(\mathbb{C}), \phi_v)_{v\in V\Gamma}$ is nuclear. 

\end{proposition}

\begin{proof}
Since nuclearity is preserved under limits, we may assume that $\Gamma$ is finite. We proceed by induction for $|V\Gamma|$. 
The base case is trivial. Assume $|V\Gamma|\geqq 2$ and take a vertex $v\in V\Gamma$ arbitrarily. 

If $v$ is adjacent to every other vertex of $V\Gamma$, 
then $M_{\Gamma}=\Mat_{n_v}(\mathbb{C})\otimes M_{\Gamma\backslash \{v \}}$ 
and we are done by the induction hypothesis. 
If not, we have the reduced amalgamated free product decomposition 
\[
(M_{\Gamma}, E_{\mathrm{Lk}v})=(M_{\mathrm{St}v}, E_1) \ast_{r, M_{\mathrm{Lk}v}} (M_{\Gamma\backslash \{v \}}, E_2). 
\] 
Here $E_1$ and $E_2$ denote the restriction of the conditional expectation 
$E_{\mathrm{Lk}v} : M_{\Gamma} \to M_{\mathrm{Lk}v}$ to each algebra. 

By \cite[Theorem 5.2]{has} and the induction hypothesis, 
it suffices to show that the GNS representation of $E_1$ contains the Jones projection and that of $E_2$ is nondegenerate. 
But the former follows from the fact that $\phi_v$ is a pure state and $(M_{\mathrm{St}v}, E_1) = (\Mat_{n_v}(\mathbb{C})\otimes M_{\mathrm{Lk}v}, \phi_v\otimes \mathrm{id})$, and the latter is trivial since $M_{\Gamma}$ is unital. 
\end{proof}

\begin{proof}[\textbf{Proof of Theorem A}]
We prove the claim first for the case when $H_v$ are bi-exact for all $v\in V\Gamma$. Let $\mathcal{G}=\{ H_{\mathrm{Lk}v}\mid v\in V\Gamma\}$. 

For each $v\in V\Gamma$, let $A_v\subset \mathcal{B}(\ell^2 H_v)$ be the $C^{\ast}$-algebra generated by $C^{\ast}_{\lambda}(H_v)$ and $S(H_v)\subset \ell^{\infty}(H_v)$. Then by Proposition 2.6 and \cite[Proposition 5.1.3]{bo}, $A_v\cong S(H_v)\rtimes_r H_v$ is nuclear. Let $A\subset B(\ell^2 H_{\Gamma})$ be the $C^{\ast}$-algebra generated by $(A_v)_{v\in V\Gamma}$. 
We prove the following two claims in accordance with \cite[Lemma 2.4]{oza2}. 

\begin{claim}
The $C^{\ast}$-algebra $A$ is nuclear. 
\end{claim}

\begin{proof}[\textbf{Proof of Claim 3.2}]
Since each $A_v$ is nuclear, there exists a family of nets of u.c.p.\@ maps $(\varphi_v^i:A_v\to \Mat_{n_v^i}(\mathbb{C}) )_i, (\psi_v^i: \Mat_{n_v^i}(\mathbb{C})\to A_v)_i$ 
such that $\|\psi_v^i\circ \varphi_v^i(x) - x\|\to 0$ holds for every $x\in A_v$.  

Moreover, by \cite[Lemma 4.1]{dp}, one can assume that there exist a net of subspaces $\delta_{e_v}\in K_v^i\subset \ell^2 H_v$ and a net of u.c.p.\@ maps $(\phi_v^i:\mathcal{B}(K_v^i)\to \Mat_{n_v^i}(\mathbb{C}) )_i$ such that $\|\psi_v^i\circ \phi_v^i\circ \mathrm{Ad}(P_{K_v^i})(x) - x\|\to 0$ holds for every $x\in A_v$. 

Since $P_v:=P_{\mathbb{C}\delta_{e_v}}\in \mathcal{K}(\ell^2 H_v) \subset A_v$, by perturbing $\psi_v^i\circ\phi_v^i$ we may assume that $\psi_v^i\circ\phi_v^i(P_v)=P_v$. 
We denote the vector state of $\delta_{e_v}$ by $\mu_v$ and the reduced graph product of $(\mathcal{B}(K_v^i), \mu_v)_{v\in V\Gamma}$ by $B^i$. Note that by Proposition 3.1 $B^i$ is nuclear. 

Then, by \cite[Corollary 4.1]{atk}, there exist the nets of u.c.p.\@ maps
$(\phi^i:A\to B^i )_i, (\psi^i: B^i \to A)_i$, which are the graph products of $(\phi_v^i)_v, (\psi_v^i)_v$ respectively, such that $\psi^i\circ \phi^i(x)\to x$ and the conclusion follows. 
\end{proof}

\begin{claim}
For every $a,b\in A$, we have $[a, JbJ]\in \mathcal{K}(\mathcal{G})$. 
\end{claim}

\begin{proof}[\textbf{Proof of Claim 3.3}]
It suffices to prove for the case when $a=f, b=u_h$
for some $v,w\in V\Gamma$ and $f\in S(H_v), h\in H_w$. 
Take an element $\xi = \xi_1\otimes\xi_2\otimes \cdots\otimes \xi_n$ from $\ell^2(H_{v_1})\otimes \ell^2(H_{v_2})^o \otimes \cdots\otimes \ell^2(H_{v_n})$ for a fixed irreducible word $v_1v_2\ldots v_n$ on $\Gamma$. (see Section 2.4). 

Then by simple calculation we get
\[
[f,Ju_hJ](\xi)= 
\begin{cases}
([f,Ju_hJ](\xi_1))\otimes \xi_2\otimes\cdots\otimes\xi_n, 
& v=w=v_1 \text{ and } v_2, \ldots, v_n\in \mathrm{Lk}(v), \\
0, & \text{otherwise},
\end{cases}
\]
and the claim follows since $[f,Ju_hJ]\in \mathcal{K}(\ell^2 H_v)$. 
\end{proof}

By Claim 3.2 and 3.3, the map
\begin{align}
\theta : A\algtensor JAJ &\to \mathcal{M}(\mathcal{K}(\mathcal{G}))/\mathcal{K}(\mathcal{G}) \\
x\otimes JyJ &\mapsto [xJyJ]
\end{align}
is well-defined and uniquely extended to the u.c.p.\@ map $\Theta$ on $A\mintensor JAJ$. And again by Claim 3.2, $\Theta$ is a nuclear and liftable u.c.p. map. By restricting $\Theta$ to $C^{\ast}_{\lambda}(H_{\Gamma})\mintensor C^{\ast}_{\rho}(H_{\Gamma})$ we get the map in Lemma 2.6 and therefore $H_{\Gamma}$ is bi-exact relative to $\mathcal{G}$ by the lemma. 

To prove for the case when $H_v$ are exact for all $v\in V\Gamma$, replace $S(H_v)$ with $\ell^{\infty}(H_v)$ in the proof above. 
Then, since $H_v$ are exact, the \Cstaralgs $A_v\cong \ell^{\infty}(H_v)\rtimes_r H_v$ are nuclear and it follows that $A=C^{\ast}\{A_v\mid v\in V\Gamma\}$ is nuclear as in Claim 3.2. Also, for Claim 3.3, 
by the same calculation as above, it follows that $[a, JbJ]\in \mathcal{K}(\{ H_{\mathrm{St}v}\}_{v\in V\Gamma})$
for every $a,b\in A$. 
Then the conclusion follows similarly. 
\end{proof}

\section{Bi-exactness for graph-wreath products}

In this section, we will prove Theorem B for the case when $H$ is bi-exact. 
To obtain the proof for the case when $H$ is a general exact group, 
replace $S(H_v)$ with $\ell^{\infty}H_v$ and
$H_{\mathrm{Lk}v}$ with $H_{\mathrm{St}v}$ in the sequel (in particular in the proof of Proposition 4.6), 
as in the proof of Theorem A above. 

Let $H$ be a bi-exact group, $\Gamma$ be a graph, and $G\curvearrowright \Gamma$ be an action of a group $G$ on $\Gamma$. 
We denote by $H_{\Gamma}$ the graph product on $\Gamma$ with each vertex group $H_v$ being $H$ and by $H_{\Gamma}\rtimes G$ the graph-wreath product. 

\begin{definition}
For $n\in \mathbb{N}$ and finite subsets $E\subset H$, $F\subset V\Gamma$, define subsets $A(E,F,n)\subset H_{\Gamma}\rtimes G$ by
\[
A(E,F,n)=\{ (h=h_1h_2\cdots h_m, g) \mid m\leqq n, h_i\in E, v_i\in F\cup gF\}
\]
where $h_1h_2\cdots h_m$ is one of the normal forms of $h\in H_{\Gamma}$ and $h_i\in H_{v_i}$ for each $i$.
(Note that this definition is well-defined.)
\end{definition}

For the proof of Theorem B, we adopt the method of \cite[Theorem 5.2]{hik}.  In fact we prove the following generalization of Theorem B. 
We set $\mathcal{G}$ to be $\{ A(E, F, n) \mid E\subset V\Gamma, F\subset H, n\} \cup\{ H_{\mathrm{Lk}v}\mid v\in V\Gamma\}$, a family of subsets of $H_{\Gamma}\rtimes G$. 

\begin{theorem}
In the above setting, assume moreover that $G$ is bi-exact relative to some family of subsets $\mathcal{S}$, all the isotropy groups of the action $G\action V\Gamma$ are finite, and that $G\action V\Gamma$ has finitely many orbits. Then, $H_{\Gamma}\rtimes G$ is bi-exact relative to $\mathcal{F}\wedge\mathcal{G}$, where $\mathcal{F}=\{ H_{\Gamma}\cdot S\mid S\in \mathcal{S}\}$. 
\end{theorem}

We divide the proof of Theorem 4.2 into three subsections. 

\subsection{A map with a good asymptotic property}

First we construct a map $m : H_{\Gamma}\rtimes G\to \ell^1(V\Gamma)$ with a good asymptotic property. We have the following lemma from \cite[Lemma 4.3]{hik}.

\begin{proposition}
Under the above setting, assume that $G\action V\Gamma$ has finite isotropy groups and finitely many orbits. 

Then, there is a function $|\cdot|_G: G\to \mathbb{R}_{\geqq 0}$, $|\cdot|_H: H\to \mathbb{R}_{\geqq 0}$, $|\cdot|_{\Gamma}: V\Gamma \to \mathbb{R}_{\geqq 0}$ such that, for every $g, g'\in G$, $h, h'\in H$, $v\in V\Gamma$, and a positive integer $n$ the following holds. 
\begin{enumerate}
\item $|gg'|_G\leqq |g|_G+|g'|_G$, $|hh'|_H\leqq |h|_H+|h'|_H$.

\item $|gv|_{\Gamma}\leqq |g|_G+|v|_{\Gamma}$.

\item The sets 
\[
\{ g\in G \mid |g|_G\leqq n\}, \{ h\in H \mid |h|_H\leqq n\}, \{ v\in V\Gamma \mid |v|_{\Gamma}\leqq n\} 
\]
are finite. 
\end{enumerate}
\end{proposition}

We fix such length functions for $G, H, V\Gamma$ as in Proposition 4.3.

\begin{definition}
For $z=(h=h_1h_2\cdots h_m, g)\in H_{\Gamma}\rtimes G$, we set
\begin{align}
m(z) &:=\sum_{v\in \mathop{supp}(h)} \min\{ |v|_{\Gamma}, |g^{-1}v|_{\Gamma}\} \delta_{v}+ \sum_{i=1}^m |h_i|_{H} \delta_{v_i}, \\
|z|_{f} &:=\| m(z)\|_1
\end{align}
where $h_1h_2\cdots h_m$ is one of the normal forms of $h\in H_{\Gamma}$, $h_i\in H_{v_i}$ for each $i$. 
(Again notice that this definition is independent of the choice of the normal form thanks to Theorem 2.11.)
\end{definition}

The proof of the following proposition is the same as that of \cite[Lemma 5.5, 5.6]{hik}. 

\begin{proposition}
The following holds for $z=(h=h_1h_2\cdots h_m, g)\in H_{\Gamma}\rtimes G$ and
$\mathcal{G}=\{ A(E, F, n) \mid E\subset V\Gamma, F\subset H, n\} \cup\{ H_{\mathrm{Lk}v}\mid v\in V\Gamma\}$; 
\begin{enumerate}
\item $|z|_f\to \infty$ when $z\to \infty/\mathcal{G}$. 

\item $\dfrac{|\mathop{supp}h|}{|z|_f} \to 0$ when $z\to \infty/\mathcal{G}$. 

\item For any $v\in V\Gamma$ and $x\in H_v$, 
\begin{align}
\| m(xz)-m(z) \|_1, \| m(zx)-m(z) \|_1 \leqq |x|_f.  
\end{align}
In particular, $\dfrac{\| m(xz)-m(z) \|_1}{|z|_f}, \dfrac{\| m(zx)-m(z) \|_1}{|z|_f} \to 0$ when $z\to \infty/\mathcal{G}$.
\item For any $k\in G$, 
\begin{align}
\| m(kz)-k.m(z) \|_1, \| m(zk)-m(z) \|_1\leqq |k|_G|\mathop{supp}h|. 
\end{align}
In particular, $\dfrac{\| m(kz)-k.m(z) \|_1}{|z|_f}, \dfrac{\| m(zk)-m(z) \|_1}{|z|_f} \to 0$ when $z\to \infty/\mathcal{G}$.
\end{enumerate}
\end{proposition}

\begin{proof}
(1) It suffices to prove that the set $\{ z\in H_{\Gamma}\rtimes G \mid |z|_f\leqq C\}$
is small relative to $\mathcal{G}$ for any positive integers $C$. 
But by the definition of $m(z)$ it easily follows that if $|(h_1h_2\cdots h_m, g)|_f\leqq C$ then the quantities $|\mathop{supp}h|$, $\min\{ |v|_{\Gamma}, |g^{-1}v|_{\Gamma}\}$, and $|h_i|_H$ are all smaller than $C$. 

Therefore we obtain
\[
\{ z\in H_{\Gamma}\rtimes G \mid |z|_f\leqq C\} \subset
A(\{ |\cdot |_H\leqq C\}, \{ |\cdot |_{\Gamma}\leqq C\}, C )
\]
and the conclusion follows by the assumption on $|\cdot|_{\Gamma}$ and $|\cdot|_H$. 

(2) It suffices to prove that, for any positive integer $C$, the set of all $z=(h,g)$ satisfying $|z|_f\leqq C |\mathop{supp} h|$ is small relative to $\mathcal{G}$. 
We claim that, if $|z|_f\leqq C |\mathop{supp} h|$, then
\[
|\mathop{supp} h| \leqq 4 |B(\Gamma,2C)|,
\] 
from which, combined with the assumption and (1), the conclusion follows. 
Here $B(\Gamma,R)=\{v\in V\Gamma \mid |v|_{\Gamma} \leqq R\}$ for a positive integer $R$. 

Suppose not. Then for every $g\in G$ one gets
\[
|B(\Gamma, 2C)\cup gB(\Gamma,2C)|< \dfrac{1}{2}|\mathop{supp}h|. 
\]
Using this inequality and the fact that $\min\{ |v|_{\Gamma}, |g^{-1}v|_{\Gamma}\}>2C$ for every $v\notin B(\Gamma, 2C)\cup gB(\Gamma,2C)$, the following inequality holds;
\begin{align}
|\mathop{supp}h| 
&\leqq |B(\Gamma, 2C)\cup gB(\Gamma,2C) | + |\mathop{supp}h \backslash (B(\Gamma, 2C)\cup gB(\Gamma,2C))| \\
&< \dfrac{1}{2}|\mathop{supp}h| + \dfrac{1}{2C}\sum_{v\in \mathop{supp}h \backslash B(\Gamma, 2C)\cup gB(\Gamma,2C)} \min\{ |v|_{\Gamma}, |g^{-1}v|_{\Gamma}\}\\
&\leqq \dfrac{1}{2}|\mathop{supp}h| + \dfrac{1}{2C}|(h,g)|_f \\
&\leqq |\mathop{supp}h|. 
\end{align}
This is a contradiction.

(3) First, note that, by the proof of Lemma 2.12, the normal form of $xh$ should be one of the following form. 
\begin{itemize}
\item $xh_1h_2\cdots h_m$, when $\|xh\| = \| h\| +1$
\item $h_1\cdots h_{i-1}(xh_i)\cdots h_n$,  when $\|xh\| = \| h\|$
\item $h_1\cdots h_{i-1}h_{i+1}\cdots h_n$, when $\|xh\| = \| h\| -1$
\end{itemize}
In particular, $\mathop{supp}(xh)$ must be either $\{v\} \cup \mathop{supp}(h)$,  $\mathop{supp}(h)$, or $\mathop{supp}(h) \backslash \{ v \}$. 
Then, by Definition 4.4, 
\begin{align}
m(xz) = 
\begin{dcases}
\sum_{w\in \mathop{supp}(h)\cup \{ v \}} \min\{ |w|_{\Gamma}, |g^{-1}w|_{\Gamma}\} \delta_{w}+ \sum_{j=1}^m |h_i|_{H} \delta_{v_i} + |x|_H\delta_v,  
& (\|xh\| = \| h\| +1), \\
\sum_{w\in \mathop{supp}(h)} \min\{ |w|_{\Gamma}, |g^{-1}w|_{\Gamma}\} \delta_{w}+ \sum_{j=1, j\neq i}^m |h_i|_{H} \delta_{v_i} + |xh_i|_H\delta_v,  
& (\|xh\| = \| h\| ), \\
\sum_{w\in \mathop{supp}(h)\backslash \{ v \}} \min\{ |w|_{\Gamma}, |g^{-1}w|_{\Gamma}\} \delta_{w}+ \sum_{j=1, j\neq i}^m |h_i|_{H} \delta_{v_i}, 
& (\|xh\| = \| h\| -1).
\end{dcases}
\end{align}
In each case, by easy calculation we deduce that $\| m(xz)-m(z) \|_1 $ is smaller than $\min\{ |v|_{\Gamma}, |g^{-1}v|_{\Gamma}\} + |x|_H = |x|_f$. 
The right version follows similarly. 

(4) Note that $kz=(\sigma_k(h), kg)$. We see, by the definition of $m(z)$, that
\begin{align}
m(kz) &= \sum_{v\in k\mathop{supp}(h)} \min\{ |v|_{\Gamma}, |g^{-1}k^{-1}v|_{\Gamma}\} \delta_{v}+ \sum_{i=1}^m |h_i|_{H} \delta_{kv_i}, \text{and}
\\
k.m(z) &= \sum_{v\in \mathop{supp}(h)} \min\{ |v|_{\Gamma}, |g^{-1}v|_{\Gamma}\} \delta_{kv}+ \sum_{i=1}^m |h_i|_{H} \delta_{kv_i}. 
\end{align}
By taking the difference, it follows that 
\begin{align}
&\| m(kz)-k.m(z) \|_1 \\
\leqq 
&\| \sum_{v\in \mathop{supp}(h)} (\min\{ |kv|_{\Gamma}, |g^{-1}v|_{\Gamma}\} - \min\{ |v|_{\Gamma}, |g^{-1}v|_{\Gamma}\}) \delta_{kv} \|_1 \\
\leqq 
&\sum_{v\in \mathop{supp}(h)} | |kv|_{\Gamma} - |v|_{\Gamma} | \\
\leqq 
&|k|_G|\mathop{supp}h|
\end{align}
and the first inequality is proven. The second one can be proven similarly using 
$zk=(h,gk)$. 
\end{proof}

\subsection{(AO) property relative to $\mathcal{G}$}
We keep the same notation as in the previous subsection. 
\begin{proposition}
Let $C=C^{\ast}_{\lambda}(H_{\Gamma}\rtimes G)$. Then, the map
\begin{align}
\theta : C\algtensor JCJ &\longto \mathcal{M}(\mathcal{K}(\mathcal{G}))/\mathcal{K}(\mathcal{G}) \\
x\otimes JyJ &\longmapsto [xJyJ]
\end{align}
is well-defined and continuous with respect to the minimal tensor norm. 
Moreover, its extension $\Theta$ to $C\mintensor JCJ$ is nuclear and has a u.c.p.\@ lift. 
\end{proposition}
\begin{proof}
As in the proof of Theorem A, let $D_v\subset \mathcal{B}(\ell^2 H_v)$ be the $C^{\ast}$-algebra generated by $C^{\ast}_{\lambda}(H_v)$ and $S(H_v)\subset \ell^{\infty}(H_v)$ for each $v\in V\Gamma$, and $D\subset B(\ell^2 H_{\Gamma})\subset B(\ell^2 (H_{\Gamma}\rtimes G))$ the $C^{\ast}$-algebra generated by $(D_v)_{v\in V\Gamma}$. 

Note that, by Claim 3.2, $D$ is nuclear. Moreover, according to the proof of Claim 3.3, for any $v,w\in V\Gamma$, $f\in S(H_v), h\in H_w$, $g\in G$ and  
$\xi = \xi_1\otimes\xi_2\otimes \cdots\otimes \xi_n$ from $\ell^2(H_{v_1})\otimes \ell^2(H_{v_2})^o \otimes \cdots\otimes \ell^2(H_{v_n})$ for a fixed irreducible word $v_1v_2\ldots v_n$ on $\Gamma$, one obtains
\begin{align}
&[f,Ju_hJ](\xi\otimes \delta_g)= 
([f,Ju_{\sigma_g(h)}J]\xi)\otimes \delta_g \\
=&
\begin{cases}
([f,Ju_hJ](\xi_1))\otimes \xi_2\otimes\cdots\otimes\xi_n, 
& v=gw=v_1 \text{ and } v_2, \ldots, v_n\in \mathrm{Lk}(v), \\
0, & \text{otherwise}.
\end{cases}
\end{align}
Since the action $G\curvearrowright \Gamma$ has finite isotropy groups, for each tuple of $v$, $w$, $f$, $h$, and $\xi$, 
the value of the expression above is nonzero for only finitely many $g\in G$. Therefore, we as well get $[a, JbJ]\in \mathcal{K}(\mathcal{G}')$ for any $a,b\in D$ in this situation, when we set $\mathcal{G}'=\{ H_{\mathrm{Lk}v}\mid v\in V\Gamma\}\subset \mathcal{G}$. 

Consequently, the map 
\begin{align}
\Psi : D\algtensor JDJ &\longto \mathcal{M}(\mathcal{K}(\mathcal{G}'))/\mathcal{K}(\mathcal{G}') \\
x\otimes JyJ &\longmapsto [xJyJ]
\end{align}
is well-defined and continuous with respect to the minimal tensor norm. 
Note that the conclusion here still holds when $\mathcal{G}'$ is replaced with any family of subsets $\mathcal{E}$ with $\mathcal{K}(\mathcal{G}')\subset \mathcal{K}(\mathcal{E})$, especially with $\mathcal{E}=\mathcal{G}$. 

Now we turn to dealing with the crossed product. Define the map
$\mu : H_{\Gamma}\rtimes G\to \ell^1(V\Gamma)$ by $\mu(z)=m(z)/|z|_f$ and let $q: \ell^{\infty}(H_{\Gamma}\rtimes G)\to \ell^{\infty}(H_{\Gamma}\rtimes G)/c_0(H_{\Gamma}\rtimes G, \mathcal{G})$ be the quotient map. Also we define 
three commutative \Cstaralgs as follows; 
\begin{align}
\mathcal{L}_{\mathcal{G}}
&= \{ f\in \ell^{\infty} (H_{\Gamma}\rtimes G)
\mid f- \prescript{x}{}{f}, f-f^x\in c_0(H_{\Gamma}\rtimes G, \mathcal{G}) \text{ for every } x\in H_{\Gamma}\}, \\
\mathcal{S}_{\mathcal{G}} 
&= \{ f\in \mathcal{L}_{\mathcal{G}}
\mid f-f^g \in c_0(H_{\Gamma}\rtimes G, \mathcal{G}) \text{ for every } g\in G\}, \\
\mathcal{C}_{\mathcal{G}}&= \mathcal{S}_{\mathcal{G}}/c_0(H_{\Gamma}\rtimes G, \mathcal{G}). 
\end{align}

Here $\prescript{x}{}{f}$, $f^x$ denote the left and right regular representation respectively. 
Note that the left regular action $G\curvearrowright \ell^{\infty}(H_{\Gamma}\rtimes G)/c_0(H_{\Gamma}\rtimes G, \mathcal{G})$ restricts to an action $\alpha: G\curvearrowright \mathcal{C}_{\mathcal{G}}$. We claim that this action is amenable. 

Let $\mu^{\ast} : \ell^{\infty}(V\Gamma) \to \ell^{\infty}(H_{\Gamma}\rtimes G)$ be the adjoint u.c.p.\@ map, i.e. $\mu^{\ast}(f)(g)= \int f d\mu(g)$. 
Then, according to Proposition 4.5, the image of the u.c.p.\@ map 
$q\circ \mu^{\ast} : \ell^{\infty}(V\Gamma) \to \ell^{\infty}(H_{\Gamma}\rtimes G)/c_0(H_{\Gamma}\rtimes G, \mathcal{G})$ is contained in $\mathcal{C}_{\mathcal{G}}$, and this map is left $G$-equivariant. Since $G$ is an exact group, it follows that the action $G\curvearrowright \ell^{\infty}(V\Gamma)$ is amenable and so is $\alpha : G\curvearrowright \mathcal{C}_{\mathcal{G}}$. 


Since $\ell^{\infty}(H_{\Gamma}\rtimes G)\subset \mathcal{M}(\mathcal{K}(\mathcal{G}))$ and $c_0(H_{\Gamma}\rtimes G, \mathcal{G})\subset \mathcal{K}(\mathcal{G})$, we have a well-defined $\ast$-homomorphism
\begin{align}
\Psi' : \mathcal{C}_{\mathcal{G}}\mintensor D\mintensor JDJ &\longrightarrow \mathcal{M}(\mathcal{K}(\mathcal{G}))/\mathcal{K}(\mathcal{G}) \\
f\otimes x\otimes JyJ &\longmapsto [fxJyJ]
\end{align}
which is again nuclear. 

Finally, one can define an action of $G\times G$ on the $C^{\ast}$-algebra $\mathcal{C}_{\mathcal{G}}\mintensor D\mintensor JDJ$ by
\[
G\times G\ni (g_1, g_2)\longmapsto 
\alpha(g_1)\otimes \mathrm{Ad}\lambda(g_1) 
\otimes \mathrm{Ad}\rho(g_2)
\]
(Here $\lambda$ and $\rho$ denote the left and right regular representation respectively). Then, since $\alpha$ is an amenable action, this action is again amenable, and the full crossed product $C^{\ast}$-algebra $(\mathcal{C}_{\mathcal{G}}\mintensor D\mintensor JDJ)\rtimes_f (G\times G)$ is nuclear and coincides with the reduced crossed product. Hence the map $\Psi'$ extends to the nuclear $\ast$-homomorphism
\begin{align}
\Theta' : (\mathcal{C}_{\mathcal{G}}\mintensor D\mintensor JDJ)\rtimes_r (G\times G) &\longrightarrow \mathcal{M}(\mathcal{K}(\mathcal{G}))/\mathcal{K}(\mathcal{G}) \\
(f\otimes x\otimes JyJ, (g_1, g_2)) &\longmapsto [fxJyJu_{g_1}Ju_{g_2}J]. 
\end{align}
Restrict this map to $C\algtensor JCJ \subset (D\mintensor JDJ)\rtimes_r (G\times G)$ and we get the desired map. 
\end{proof}

\subsection{Putting relative bi-exactness together}

The argument here is identical to the proof of \cite[Lemma 3.13]{hik}. However we present it for completeness. 

\begin{proof}[Proof of Theorem 4.2]
Recall that $G$ is bi-exact relative to $\mathcal{S}$, and $\mathcal{F}=\{ H_{\Gamma}\cdot S\mid S\in \mathcal{S}\}$. 
We set $\tilde{G}=G\times G$. 

Let $\mathcal{K}=\mathcal{K}(\mathcal{F}\wedge\mathcal{G})=\mathcal{K}(\mathcal{F})\cap\mathcal{K}(\mathcal{G})$. 
We define three \Cstaralgs $\mathscr{B}\subseteq\mathscr{A}\subseteq\mathscr{M}$ 
of $\mathcal{B}(\ell^2(H_{\Gamma}\rtimes G))$ by
\begin{align}
\mathscr{B}&=C^{\ast} \{C^{\ast}(H_{\Gamma}), 
JC^{\ast}(H_{\Gamma})J, 1\otimes \ell^{\infty}G\}, \\
\mathscr{A}&=C^{\ast} \{C^{\ast}(H_{\Gamma}), 
JC^{\ast}(H_{\Gamma})J, \mathcal{L}_{\mathcal{G}}\}, \\
\mathscr{M}&=C^{\ast} \{C^{\ast}(H_{\Gamma}), 
JC^{\ast}(H_{\Gamma})J, \ell^{\infty}(H_{\Gamma}\rtimes G)\}. 
\end{align}
Each of these algebras have an action of $\tilde{G}$: 
\[
\tilde{G}\ni (g_1, g_2)\longmapsto \mathrm{Ad}(u_{g_1}Ju_{g_2}J). 
\]

We often simply write $\mathscr{A}/(\mathcal{K}\cap\mathscr{A})$
by $\mathscr{A}/\mathcal{K}$ (similarly for $\mathscr{B}$, $\mathscr{M}$ instead of $\mathscr{A}$
and $\mathcal{K}(\mathcal{F})$, $\mathcal{K}(\mathcal{G})$ instead of $\mathcal{K}$). 

We claim that
\begin{align}
\Phi : C\algtensor JCJ &\longto \mathcal{M}(\mathcal{K})/\mathcal{K} \tag{4.1}  \\
x\otimes JyJ &\longmapsto [xJyJ] 
\end{align}
is min-continuous, nuclear, and its extension has a u.c.p.\@ lift $C\mintensor JCJ\to \mathcal{M}(\mathcal{K})$, from which the result follows by Proposition 2.6. 
We prove this by proving the three claims below. 
We will denote $\rtimes_r$ by the reduced crossed product \Cstaralg
and $\rtimes_f$ by the full crossed product \Cstaralg.

\begin{claim}
There are $\ast$-homomorphisms $\varphi$, $\pi$, $\psi$ between the short exact sequences
\[
\begin{tikzcd}
  0 \arrow[r] 
    & ((\mathcal{K}(\mathcal{F})\cap \mathscr{A})/\mathcal{K})\rtimes_r \tilde{G} \arrow[r] \arrow[d, "\varphi"]
    & (\mathscr{A}/\mathcal{K})\rtimes_r \tilde{G} \arrow[r] \arrow[d, "\pi"]
    & (\mathscr{A}/\mathcal{K}(\mathcal{F}))\rtimes_r \tilde{G} \arrow[r] \arrow[d, "\psi"]
    & 0 \\
  0 \arrow[r]
    & ((\mathcal{K}(\mathcal{F})\cap \mathscr{M})/\mathcal{K})\rtimes_f \tilde{G} \arrow[r]
    & (\mathscr{M}/\mathcal{K})\rtimes_f \tilde{G} \arrow[r]
    & (\mathscr{M}/\mathcal{K}(\mathcal{F}))\rtimes_f \tilde{G} \arrow[r]
    & 0
\end{tikzcd}
\]
which are induced by the inclusion $\mathscr{A}\subseteq \mathscr{M}$. 
\end{claim}
\begin{proof}[Proof of Claim 4.7]
First note that the left-right regular action 
$\tilde{G} \action \ell^{\infty}G/c_0(G;\mathcal{S})$ 
is amenable by Proposition 2.6. 
Then the existence of $\psi$ follows from the $\tilde{G}$-equivariant inclusion
\[
1 \otimes \ell^{\infty}G/c_0(G;\mathcal{S})
\subset \ell^{\infty}(H_{\Gamma}\rtimes G/c_0(H_{\Gamma}\rtimes G;\mathcal{F}) 
\subset \mathscr{B}/\mathcal{K}(\mathcal{F})
\subset \mathscr{A}/\mathcal{K}(\mathcal{F}). 
\]

For the existence of $\varphi$, we make use of the following 
$\tilde{G}$-equivqriant inclusion 
\begin{align}
(\mathcal{K}(\mathcal{F})\cap \mathscr{A}) / \mathcal{K}
&= \mathcal{I}(\mathcal{F}) / (\mathcal{I}(\mathcal{F})\cap \mathcal{I}(\mathcal{G})) \\
&\cong (\mathcal{I}(\mathcal{F})+\mathcal{I}(\mathcal{G})) / \mathcal{I}(\mathcal{G}) \\
&\lhd \mathscr{A} / \mathcal{I}(\mathcal{G}), 
\end{align}
where $\mathcal{I}(\mathcal{F})=\mathcal{K}(\mathcal{F})\cap \mathscr{A}\lhd \mathscr{A}$. Note that the last one canonically includes 
$\mathcal{L}_{\mathcal{G}} / c_0(H_{\Gamma}\rtimes G; \mathcal{G})$ where $\mathcal{L}_{\mathcal{G}}$ is as in the proof of the previous proposition. 
Since there is a $\tilde{G}$-equivariant u.c.p.\@ map
\begin{align}
\mathcal{C}_{\mathcal{G}} \otimes \overline{\mathcal{C}_{\mathcal{G}}}
&\longto \mathcal{L}_{\mathcal{G}}/ c_0(H_{\Gamma}\rtimes G; \mathcal{G}) \\
f\otimes g &\longmapsto fg
\end{align}
where $\mathcal{C}_{\mathcal{G}}$ is endowed with trivial right $G$-action and $\overline{\mathcal{C}_{\mathcal{G}}}$ is the \Cstaralg $\mathcal{C}_{\mathcal{G}}$ with its left and right $\tilde{G}$-action flipped; $g\cdot \overline{f} \cdot h = h^{-1}fg^{-1}$ for $\overline{f}\in \overline{\mathcal{C}_{\mathcal{G}}}$. 
Given that the action $G\action \mathcal{C}_{\mathcal{G}}$ is amenable, it follows that $\tilde{G}\action \mathscr{A} / \mathcal{I}(\mathcal{G})$
is an amenable action
and therefore its reduced and full crossed product coincide. 
By passing to the closed ideal generated by $(\mathcal{K}(\mathcal{F})\cap \mathscr{A}) / \mathcal{K}$ and $\tilde{G}$, it also follows that the reduced and full crossed product of 
$\tilde{G}\action (\mathcal{K}(\mathcal{F})\cap \mathscr{A}) / \mathcal{K}$
coincide. 
This verifies the existence of $\varphi$. 

Finally the existence of $\pi$ follows from the five lemma. 
\end{proof}

\begin{claim}
In the above setting, $\varphi$ and $\psi\vert_{(\mathscr{B}/\mathcal{K})\rtimes_r \tilde{G}}$ are nuclear maps. 
\end{claim}
\begin{proof}[Proof of Claim 4.8]
Note again that, since the left-right regular action of $\tilde{G}$ on 
$\mathcal{L}_{\mathcal{G}} / c_0(H_{\Gamma}\rtimes G; \mathcal{G})$ 
and $\ell^{\infty}G/c_0(G;\mathcal{S})$ are amenable, 
the reduced crossed products 
$(\mathscr{A} / \mathcal{I}(\mathcal{G}))\rtimes_r \tilde{G}$ 
and $(\mathscr{B}/\mathcal{K})\rtimes_r \tilde{G}$ are nuclear. 
The second assertion follows directly from the latter. 
For the first one, just recall that the \Cstaralg
$((\mathcal{K}(\mathcal{F})\cap \mathscr{A}) / \mathcal{K})\rtimes_r \tilde{G}$
is $\ast$-isomorphic to a closed ideal in
$(\mathscr{A} / \mathcal{I}(\mathcal{G}))\rtimes_r \tilde{G}$
and hence nuclear. 
\end{proof}

\begin{claim}
The inclusion $\mathscr{B}\subseteq \mathscr{M}$ extends to a nuclear u.c.p.\@ map
\[
\pi' : (\mathscr{B}/\mathcal{K})\rtimes_r \tilde{G}
\longrightarrow (\mathscr{M}/\mathcal{K})\rtimes_f \tilde{G}. 
\]
\end{claim}
\begin{proof}[Proof of Claim 4.9]
It suffices to prove that the restriction of $\pi$ 
to $(\mathscr{B}/\mathcal{K})\rtimes_r \tilde{G}$ is nuclear. 

Since the ideal 
$\mathcal{K}(\mathcal{F})\cap \mathscr{B}\lhd \mathcal{K}(\mathcal{F})\cap \mathscr{M}$ has an approximation unit $\{ 1_{H_{\Gamma}S} = 1\otimes 1_S\mid S
\text{: small relative to }
\mathcal{S}\}$, this follows from Lemma 4.10 below. 
\end{proof}

Recall that, by the proof of Proposition 4.6, there is a nuclear $\ast$-homomorphism with a u.c.p.\@ lift
\begin{align}
\Psi : C^{\ast}_{\lambda}(H_{\Gamma})\mintensor
JC^{\ast}_{\lambda}(H_{\Gamma})J &\longto \mathcal{M}(\mathcal{K}(\mathcal{F}\wedge\mathcal{G}))/\mathcal{K}(\mathcal{F}\wedge\mathcal{G}) \\
x\otimes JyJ &\longmapsto [xJyJ]
\end{align}
since $\mathcal{K}(\mathcal{G}')\subset\mathcal{K}(\mathcal{F}\wedge\mathcal{G})$. 
By Claim 4.7, we obtain a chain of u.c.p.\@ maps
\begin{align}
C\mintensor JCJ &\xrightarrow{\cong} (C^{\ast}_{\lambda}(H_{\Gamma})\mintensor
JC^{\ast}_{\lambda}(H_{\Gamma})J)\rtimes_r \tilde{G} \\
&\xrightarrow{\Psi} (\mathscr{B}/\mathcal{K})\rtimes_r \tilde{G} \\
&\xrightarrow{\pi} (\mathscr{M}/\mathcal{K})\rtimes_f \tilde{G} \\
&\xrightarrow{f} \mathcal{M}(\mathcal{K})/\mathcal{K}
\end{align}
where $f$ is the $\ast$-homomorphism defined by 
$f\vert_{\mathscr{M}/\mathcal{K}}=\mathrm{id}_{\mathscr{M}/\mathcal{K}}$ 
and $f(u_{(g_1, g_2)})=u_{g_1}Ju_{g_2}J$. Then the composition of all of the maps extends the map $\Phi$ above
and the claim follows since $\pi$ is nuclear. 
\end{proof}

\begin{lemma}
Suppose that there is two short exact sequences of \Cstaralgs and u.c.p.\@ maps
$\varphi : I\to J$, $\pi : A\to B$ and $\psi : A/I\to B/J$ 
\[
\begin{tikzcd}
  0 \arrow[r] 
    & I \arrow[r] \arrow[d, "\varphi"]
    & A \arrow[r] \arrow[d, "\pi"]
    & A/I \arrow[r] \arrow[d, "\psi"]
    & 0 \\
  0 \arrow[r]
    & J \arrow[r]
    & B \arrow[r]
    & B/J \arrow[r]
    & 0
\end{tikzcd}
\]
with $A$ being exact and $\varphi$, $\psi$ being nuclear. If there is an approximation unit $(a_i)_i$ in $I$ such that $(\varphi(a_i))_i$ is again an approximation unit in $J$, then $\pi$ is nuclear as well. 
\end{lemma}
\begin{proof}
Passing to the double duals, it suffices to prove that $\pi^{\bidual} : A^{\bidual}\to B^{\bidual}$ is weakly nuclear. 
Note that $\varphi^{\bidual}$ is unital by the assumption. 
Then, since we have the decompositions
$A^{\bidual} \cong I^{\bidual}\oplus (A/I)^{\bidual}$ and 
$B^{\bidual} \cong J^{\bidual}\oplus (B/J)^{\bidual}$, 
and since the restriction of $\pi^{\bidual}$ to the first component coincides with $\varphi^{\bidual}$, it follows that $\pi^{\bidual}=\varphi^{\bidual}\oplus \psi^{\bidual}$. Then the result holds since 
$\varphi^{\bidual}$ and  $\psi^{\bidual}$ are weakly nuclear. 
\end{proof}

\section{Rigidity and primeness for graph-wreath products}

In this section we study the property of graph-wreath products of von Neumann algebras. 
We will consider the following situation in the sequel. 
\begin{notation}
\begin{enumerate}
\item Let $G$ be an infinite bi-exact group and $H$ an infinite exact group. 
\item Let $B=LH$ be the group von Neumann algebra of $H$. 
\item Let $\Gamma$ be an infinite locally finite graph, and assume that there is an action $G\curvearrowright \Gamma$ with finitely many orbits such that the isotropy groups of its vertices are all finite. 
\item We denote by $B_{\Gamma}$ the graph product algebra with each vertex algebra being $B$, and by $M=B_{\Gamma}\rtimes G$ the crossed product associated with the action $\sigma:G\curvearrowright B_{\Gamma}$. 
\end{enumerate}
\end{notation}
We start by proving the following equivariant version of Lemma 2.13 (4)'.  

\begin{lemma}
Let $P\subset pMp$ be a von Neumann subalgebra such that there exists a finite subset $E\subset V\Gamma$ such that $P\prec^s_M B_E$ and $P\nprec_M B_{E'}$ for any $E'\subsetneq E$. 

If there is a finite subset $F\subset V\Gamma$ such that $P\prec_M B_F$, then there is an element $g\in G$ such that $gE\subset F$. 
\end{lemma}
\begin{proof}
We prove it first for the case when $p\in B_E$ and $P\subset pB_Ep$. Suppose that the conclusion does not hold. Then, by Lemma 2.4 and 2.13 (4)', there is a net of unitaries $u_i\in \mathcal{U}(P)$ such that, for any $g\in G$ and $a,b \in B_\Gamma$, $\|E_{B_{gF}}(a^{\ast}u_ib)\|_2$ converges to $0$. 

We claim
$\|E_{B_F}(a^{\ast}u_ib)\|_2 \to 0$
for every $a,b\in M$, which immediately contradicts the assumption. 
It suffices to prove for the case $a^{\ast}=u_gx, b=yu_h$ where $x, y\in B_{\Gamma}$. Then
\begin{align}
\|E_{B_F}(a^{\ast}u_ib)\|_2&=\|E_{B_F}(\sigma_g(xu_iy)u_{gh})\|_2 \\
&= \delta_{gh,e} \|\sigma_g(E_{B_{g^{-1}F}}(xu_iy))\|_2\to 0
\end{align}
and the claim is proven. 

For the general case we proceed as in \cite[Claim 3.6]{hi}. 
Since $P\prec_M B_F$, there exist nonzero projections $r\in P$, $q\in B_F$, a normal $\ast$-homomorphism $\theta:rPr\to qB_Fq$, and a partial isometry $v\in rMq$ such that $v\theta(x)=xv$ for every $x\in rPr$. 
As $vv^{\ast}\in (rPr)'\cap rMr = (P'\cap pMp)r$, there is a projection $r'\in P'\cap pMp$ such that $vv^{\ast} = rr'$. By the assumption and Lemma 2.3, 
there exist nonzero projections $s\in rPrr'$, $q'\in B_E$, a normal $\ast$-homomorphism $\varphi:sPs\to q'B_Eq'$, and a partial isometry $w\in sMq'$ such that $w\varphi(x)=xw$ for every $x\in sPs$
and $Q:=\varphi(sPs)\nprec B_{E'}$ for each $E'\subsetneq E$. 

Combined together, we get $w^{\ast}vv^{\ast}w=w^{\ast}rr'w=w^{\ast}w \neq 0$ and, for every $x\in sPs$, 
\[ v^{\ast}w \varphi(x) = v^{\ast}xw = \theta(x) v^{\ast}w.  \]
Therefore $Q \prec_M B_F$. Now we can apply the first half. 
\end{proof}


The following lemma is a variant of \cite[Lemma 7.2]{im}. 
\begin{lemma}
Let $P\subset pMp$ be a von Neumann subalgebra such that there exists a finite subset $E\subset V\Gamma$ such that $P\prec^s_M B_E$ and $P\nprec_M B_{E'}$ for any $E'\subsetneq E$. 

Then, for every $\varepsilon>0$, there is a projection $q\in P'\cap pMp$ with $\tau(q)>1-\varepsilon$ such that $\mathcal{QN}_{qMq}(Pq)'' \prec^s B_{E\cup\mathrm{Lk}E}\rtimes G^E$, where $G^E = \{ g\in G \mid gE=E\}$ is the isotropy group of $E$.  
\end{lemma}
\begin{proof}
As in the proof of the previous lemma, we first prove it 
for the case when $p\in B_E$ and $P\subset pB_Ep$. 
Take $x\in \mathcal{QN}_{pMp}(P)$ and $x_1, \ldots, x_k \in pMp$ such that
$Px\subset \sum_i x_iP$. 
Let $x=\sum_{g\in G} x^gu_g$ and $x_i=\sum_{g\in G} x_i^gu_g$ be their Fourier expansions. 
Then, for every $g\in G$, we get 
$Px^g\subset\sum_i x_i^g\sigma_g(P)\subset \sum x_i^gB_{gE}$ by a simple calculation. 
If $x^g \neq 0$, this shows that there is a right $B_{gE}$-finite $P$-$B_{gE}$-bimodule 
$\overline{\mathop{span}} Px^gB_{gE} \subset L^2(B_{\Gamma})$ and therefore 
$P\prec_{B_{\Gamma}} B_{gE}$. By Lemma 2.13 (4)', it follows that $E\subset gE$, and since $E$ is finite, $E=gE$ and $g\in G^E$. Moreover, by Lemma 2.13 (3)', $x^g$ belongs to $B_{E\cup\mathrm{Lk}E}$. 
Finally we get $\mathcal{QN}_{pMp}(P)\subset B_{E\cup\mathrm{Lk}E}\rtimes G^E$. 

The general case follows as in \cite[Lemma 7.2 (ii)]{im}. 
Take a positive integer $n$, 
a normal $\ast$-homomorphism $\pi : P\to \Mat_n(B_E)$, 
and a partial isometry $v\in \Mat_{1,n}(M)$ such that 
$v\pi(x)=xv$ for every $x\in P$ 
and $\tau(vv^{\ast})>1-\varepsilon$
as in Theorem 2.1. 
Using Lemma 2.3, one may moreover assume that 
$\pi(P) \nprec_{\Mat_n(M)} \Mat_n(B_{E'})$ for any proper subsets $E'\subsetneq E$. 
Then, by an argument similar to the first half of this proof, 
it follows that 
\[
\mathcal{QN}_{r\Mat_n(M)r}(\pi(P)r)\subset \mathcal{QN}_{\pi(p)\Mat_n(M)\pi(p)}(\pi(P))\subset \Mat_n(B_{E\cup\mathrm{Lk}E}\rtimes G^E), 
\]
where $r=v^{\ast}v\in \pi(P)'\cap \pi(p)\Mat_n(M)\pi(p)$. 
By letting $q=vv^{\ast}\in P'\cap pMp$, we have
\[
v^{\ast} \mathcal{QN}_{qMq}(Pq) v = \mathcal{QN}_{r\Mat_n(M)r}(\pi(P)r)
\]
and the conclusion is obtained.

\end{proof}

\subsection{Primeness for graph-wreath products}
Now we prove the primeness result for the graph-wreath products. 

\begin{proof}[Proof of Theorem E]
Assume that there is a tensor product decomposition of the \twoone factor  $M=B_{\Gamma}\rtimes G=P\overline{\otimes} Q$. 
We may assume that $P$ is nonamenable and $Q$ is diffuse by Lemma 2.14 (3). 

Then, by Theorem B, there exist a projection $p\in Q'\cap M=P$
and a vertex $v$ such that $Qp\prec^s_M B_{E}$, where $E=\mathrm{St}v$ is a finite set by the assumption.  
By shrinking $p$ and $E$ if necessary, 
one may take a projection $p$ and a non-empty finite subset $E\subset V\Gamma$ 
such that $Qp\prec^s_M B_{E}$ and $Qp\nprec_M B_{E'}$ for all proper subsets $E' \subsetneq E$. 

But then, by Lemma 5.3, there is a projection $q\in pPp$ such that 
$qMq=\mathcal{N}_{qMq}(Qq)''\prec_M B_{E\cup \mathrm{Lk}E}\subset B_{\Gamma}$, 
since $G^E$ is finite. This contradicts $G$ being infinite. 
\end{proof}

\subsection{Rigidity of graphs for graph-wreath products}

Let $(H_1, G_1\action \Gamma_1)$ and $(H_2, G_2\action \Gamma_2)$ 
be the tuples satisfying the conditions in Notation 5.1. 

For clarity, unless otherwise specified, we will write by $v$ an element in $V\Gamma_1$ and by $w$ that in $V\Gamma_2$, occasionally with indices.  

We fix the notation $B=LH_1$, $C=LH_2$, $M=B_{\Gamma_1}\rtimes G_1=L((H_1)_{\Gamma_1}\rtimes G_1)$, $N=C_{\Gamma_2}\rtimes G_2=L((H_2)_{\Gamma_2}\rtimes G_2)$, 
which is different from Notation 5.1, 
to avoid the overuse of indices. 
Also, let $S$ and $T$ be the (finite) sets of representatives of $G_1\backslash V\Gamma_1$ and $T=G_2\backslash V\Gamma_2$, respectively. 
Also recall that a graph $\Gamma$ is untransvectable if and only if 
$\mathrm{Lk}_{\Gamma}v \nsubseteq \mathrm{St}_{\Gamma}v' $ holds for each pair of vertices $v\neq v'\in V\Gamma$, 
and note that an infinite connected untransvectable graph has no vertices with degree less than 2. 

Finally suppose that there is a \twoone factor $P$ and 
there are projections $p,q\in P$
such that there are isomorphisms $M\cong pPp$, $N\cong qPq$. 
We identify $M$ with $pPp$ and $N$ with $qPq$ by these isomorphisms. 

We use the method of \cite[Proposition 6.4]{hi} to prove the following;

\begin{theorem}
In the setting above, assume moreover that
\begin{itemize}
\item $\Gamma_i$ is a untransvectable graph with $\mathrm{girth}\Gamma_i\geqq 5$, 
\item The action $G_i\action V\Gamma_i$ satisfies the following condition;
for every vertex $v\in V\Gamma_i$ and $g\in G_i$, if $gv'=v'$ holds for any $v'\in \mathrm{St}_{\Gamma_i}v$, then $g=e$, 
\end{itemize}
for $i=1,2$. 
Then, there are projections
$(p_{v,w})_{v\in S, w\in T}$ and $(q_{w,v})_{v\in S, w\in T}$ such that
the following holds for every $v\in S$ and $w\in T$. 
\begin{enumerate}
\item $p_{v,w}\in \mathcal{Z}(B'_{\mathrm{St}v}\cap pPp)$, 
$q_{w,v}\in \mathcal{Z}(C'_{\mathrm{St}w}\cap qPq)$. \\
\item $\sum_{w'\in T} p_{v,w'}=p$, $\sum_{v'\in S} q_{w,v'}=q$  \\
\item $B_{\mathrm{St}v}p_{v,w} \prec^s_P C_{\mathrm{St}w}q_{w,v}$, 
$C_{\mathrm{St}w}q_{w,v} \prec^s_P B_{\mathrm{St}v}p_{v,w}$
\end{enumerate}
\end{theorem}

\begin{proof}
First of all, note that $\mathrm{Lk}_{\Gamma_i}(\mathrm{St}_{\Gamma_i}v)=\emptyset$ 
and $B_{\mathrm{Lk}v}$ is a nonamenable \twoone factor 
for any $v\in\Gamma_i$, since $\Gamma_i$ is untransvectable. 

For each pair $(v,w)\in S\times T$, 
let $p_{v,w}\in \mathcal{Z}(B'_{\mathrm{St}v}\cap pPp)$ and $q_{w,v}\in \mathcal{Z}(C'_{\mathrm{St}w}\cap qPq)$ be the maximal projection 
such that $B_{\mathrm{St}v}p_{v,w}\prec^s C_{\mathrm{St}w}$ and 
$C_{\mathrm{St}w}q_{w,v}\prec^s B_{\mathrm{St}v}$. 
(These projections can be zero.)

We claim first that $\bigvee_{w'\in T} p_{v,w'}=p$ and $\bigvee_{v'\in S} q_{w,v'}=q$
for each $v$ and $w$. 
Take an arbitrary vertex $v\in S$ 
and a projection $r\in \mathcal{Z}(B'_{\mathrm{St}v}\cap pPp)$. 
Then, since $r\in\mathcal{Z}(B_v\rtimes G_0^{\mathrm{St}v})=\mathcal{Z}(B_v)$ by Lemma 2.14 (5), 
$B_vr$ and $B_{\mathrm{Lk}v}r$ are two commuting subalgebras the latter of which is nonamenable. 
Then, by the relative bi-exactness in Theorem B, 
there exists a vertex $w\in V\Gamma_2$ such that $B_{v}r\prec C_{\mathrm{St}w}$. 
By taking a conjugation by the canonical unitaries if necessary, one may take $w\in T$. 

Set $E=\mathrm{St}w$. By shrinking $r$ and $E$ if necessary, 
we may take $r\in \mathcal{Z}(B'_v\cap pPp)$ and $E\subset \mathrm{St}w$ so that $B_vr\prec^s C_E$ and $B_vr\nprec C_{E'}$ hold
for any proper subsets $E'\subset E$. 
By Lemma 5.3, there is a projection $p_0\in (B_vr)'\cap rMr= r((\mathcal{Z}(B_v)\barotimes B_{\mathrm{Lk}v})\rtimes G^v)r$ such that
$\mathcal{QN}_{p_0Mp_0}(B_vp_0)''\prec C_{E\cup \mathrm{Lk}E}$. 
But since $p_0\mathcal{QN}_M(B_v)''p_0$ is contained in $\mathcal{QN}_{p_0Mp_0}(B_vp_0)''$
by \cite[Lemma 3.5]{pop1}, it follows that 
$p_0(B_{\mathrm{St}v}\rtimes G^v)p_0\prec C_{E\cup \mathrm{Lk}E}$
together with Lemma 2.14 (4). 

Take a projection $p'\in (B_{\mathrm{St}v}\rtimes G^v)'\cap M\subset B_{\mathrm{St}v}\rtimes G^v$
such that $p_0(B_{\mathrm{St}v}\rtimes G^v)p_0p'\prec^s C_{E\cup \mathrm{Lk}E}$, 
and $z\in \mathcal{Z}(B_{\mathrm{St}v}\rtimes G^v)=\mathcal{Z}(B_v)$ be the central support of $p_0p'$. 
Note that here we have made use of Lemma 2.14 (5) again. 
Then, it follows that 
$B_{\mathrm{St}v}z\subset (B_{\mathrm{St}v}\rtimes G^v)z\prec C_{E\cup \mathrm{Lk}E}$. 
Furthermore, since $E\subset \mathrm{St}w$ for some $w\in T$ and $\mathrm{girth}\Gamma_2 \geqq 5$, either $E$ or $\mathrm{Lk}E$ is a singleton and 
therefore there is a $w'\in V\Gamma_2$ such that $E\cup \mathrm{Lk}E\subset \mathrm{St}w'$. 
This verifies our claim. 

Next, it is easily seen that 
$B_{\mathrm{St}v}p_{v,w}\prec^s C_{\mathrm{St}w}q_{w,v}$. 
Indeed, if not, by the first part of this proof, there is a $v\neq v'\in S$ 
such that $B_{\mathrm{St}v}p_{v,w}\prec C_{\mathrm{St}w}q_{w,v'}$. 
But then since $C_{\mathrm{St}w}q_{w,v'} \prec^s B_{\mathrm{St}v'}$, 
it follows that $B_{\mathrm{St}v}\prec_P B_{\mathrm{St}v'}$ by Lemma 2.2, 
and $\mathrm{St}_{\Gamma_1}(gv) \subset \mathrm{St}_{\Gamma_1}v'$ for some $g\in G_1$ by Lemma 5.2. 
But this contradicts with the assumption that $\Gamma_1$ is untransvectable. 
Similarly one gets $C_{\mathrm{St}w}q_{w,v} \prec^s B_{\mathrm{St}v}p_{v,w}$. 

Finally we prove that the condition (2) holds. Suppose that there are vertices 
$v\in S$ and $w\neq w'\in T$ such that $p'=p_{v,w}p_{v,w'}\neq 0$.
If $C_{\mathrm{St}w}q_{w,v}\prec B_{\mathrm{St}v}p'$, 
then together with $B_{\mathrm{St}v}p'\prec^s C_{\mathrm{St}w'}q_{w',v}$
it follows that $C_{\mathrm{St}w}q_{w,v}\prec C_{\mathrm{St}w'}q_{w',v}$, 
but this again contradicts with that $\Gamma_2$ is untransvectable. 
If $C_{\mathrm{St}w}q_{w,v}\nprec B_{\mathrm{St}v}p'$, 
then by the first claim of this proof, 
$C_{\mathrm{St}w}q_{w,v}\prec^s B_{\mathrm{St}v}(p_{v,w}-p')$
and therefore $B_{\mathrm{St}v}p'\prec^s B_{\mathrm{St}v}(p_{v,w}-p')$. 
But,
since $\mathcal{Z}(B'_{\mathrm{St}v}\cap pPp)
=\mathcal{Z}(B_{\mathrm{St}v})\overline{\otimes}\mathcal{Z}(LG_0^{\mathrm{St}v})
=\mathcal{Z}(B_{v})$ 
by Lemma 2.14 (5), 
this contradicts with $p', p_{v,w}-p'\in \mathcal{Z}(B'_{\mathrm{St}v}\cap pPp)$
and Lemma 2.14 (4), which states that 
every right $B_{\mathrm{Lk}v}$-finite
$B_{\mathrm{St}v}$-$B_{\mathrm{St}v}$-bimodule in $L^2M$ 
is contained in $L^2(B_{\mathrm{St}v}\rtimes G^{\mathrm{St} v})=L^2(B_{\mathrm{St}v}\rtimes G^{v})$. 
\end{proof}

Now we are ready for the proof of Theorem C. 

\begin{proof}[Proof of Theorem C]
In the setting of Theorem 5.4, assume moreover that $H_1$, $H_2$ are i.c.c.\@ and 
$G_1\action V\Gamma_1$, $G_2\action V\Gamma_2$ are free. 

Since the vertex algebras $B_v$ and $C_w$ are factors,
by Lemma 2.14 (5), it follows that $\mathcal{Z}(B'_{\mathrm{St}v}\cap pPp)=\C p$ and
$\mathcal{Z}(C'_{\mathrm{St}w}\cap qPq)=\C q$. for each $v\in S$ and $w\in T$. 

Take $v\in V\Gamma$ 
Then, by Theorem 5.4, it follows that for each $v\in S$ there is a unique $w\in T$
such that $B_{\mathrm{St}v}\prec^s C_{\mathrm{St}w}$ and vice versa. This gives a 
bijection between $G_1\backslash V\Gamma_1$ and $G_2\backslash V\Gamma_2$ and 
therefore verifies the first half of the theorem.

For the latter part we make use of the amalgamated free product presentation $B_{\mathrm{St}v}=\ast_{v'\in \mathrm{Lk}v, B_v} (B_v\otimes B_{v'})$,
which is due to the assumption $\mathrm{girth}\Gamma_i\geqq 5$. 
Suppose that $v\in S$ and $w\in T$ satisfy
$B_{\mathrm{St}v}\prec^s C_{\mathrm{St}w}$ and 
$C_{\mathrm{St}w}\prec^s B_{\mathrm{St}v}$. 
Then, there are 
a positive integer $n$, a normal $\ast$-homomorphism $\pi : B_{\mathrm{St}v}\to \Mat_n(C_{\mathrm{St}w})$, and a partial isometry $u\in \Mat_{1,n}(P)$ such that 
$u\pi(x)=xu$ for every $x\in B_{\mathrm{St}v}$. 
By shrinking the domain of $u$ if necessary, one may
moreover assume that $f:=\pi(p)=E_{\Mat_n(C_{\mathrm{St}w})}(u^{\ast}u)$. 

Since $\pi(B_v\overline{\otimes} B_{v'})$ has Property (T)
for every $v'\in \mathrm{Lk}v$, 
by Theorem 2.15, there is $w'\in \mathrm{Lk}w$
such that 
$\pi(B_v\overline{\otimes} B_{v'})\prec_{\Mat_n(C_{\mathrm{St}w})} C_w\overline{\otimes} C_{w'}$. 
Then, there are 
a positive integer $m$, a normal $\ast$-homomorphism 
$\varphi : \pi(B_v\overline{\otimes} B_{v'})
\to \Mat_{mn}(C_w\overline{\otimes} C_{w'})$, and a partial isometry 
$u'\in f\Mat_{n,mn}(C_{\mathrm{St}w})$ such that 
$u'\varphi(y)=yu'$ for every $y\in \pi(B_v\overline{\otimes} B_{v'})$. 
Combined together, one can see that $uu'\neq 0$ by the assumption on $f$ and $u$, 
and the $\ast$-homomorphism $\varphi\circ \pi : B_v\overline{\otimes} B_{v'} \to \Mat_{mn}(C_w\overline{\otimes} C_{w'})$ satisfies
$uu'\varphi(\pi(x))= u\pi(x)u'=xuu'$. 

Therefore, for every edge $(v,v')\in E\Gamma_1$, 
there is an edge $(w,w')\in E\Gamma_2$ such that
$B_v\overline{\otimes} B_{v'} \prec^s_P C_w\overline{\otimes} C_{w'}$
(The relation $\prec^s$ follows again because of Lemma 2.14 (5)). 
What is left to be proven is that this correspondence gives an isomorphism between
the quotient graphs. 

Suppose that $v\in S$, $w\in T$, $v'\in \mathrm{Lk}v$, and $w'\neq w''\in \mathrm{Lk}w$
satisfy 
$B_{\mathrm{Lk}v}\prec^s C_{\mathrm{Lk}w}$, 
$B_v\overline{\otimes} B_{v'} \prec^s_P C_w\overline{\otimes} C_{w'}$, and
$C_w\overline{\otimes} C_{w''} \prec^s_P B_v\overline{\otimes} B_{v'}$. 
But this instantly gives 
$C_w\overline{\otimes} C_{w''}\prec^s C_w\overline{\otimes} C_{w'}$, 
and by Lemma 5.2 there is an element $g\in G_2$ such that 
$g\{ w, w' \} = \{w, w''\}$. 
Since $w'\neq w''$ and the action is free, 
it follows that $(gw,gw')=(w'', w)$ and therefore two edges $(w,w'), (w,w'')\in E\Gamma_2$ are conjugate modulo action. This verifies the claim. 
\end{proof}

As a side note, under the assumption of $H_1$, $H_2$ being bi-exact, 
we can claim that one can obtain a correspondence 
between the subalgebras $B_{\mathrm{Lk}v}$ and $C_{\mathrm{Lk}w}$
instead of Theorem 5.4.

\begin{theorem}
Assume that, instead of the assumption in Theorem 5.4, 
\begin{itemize}
\item $H_i$ is nonamenable and bi-exact, 
\item $\Gamma_i$ is a rigid graph i.e. $\mathrm{Lk}_{\Gamma_i}(\mathrm{Lk}_{\Gamma_i}v)=\{ v \}$ for every $v\in V\Gamma_i$, 
\item The action $G_i\action V\Gamma_i$ satisfies the following condition;
for every vertex $v\in V\Gamma_i$ and $g\in G_i$, if $gv'=v'$ holds for any $v'\in \mathrm{St}_{\Gamma_i}v$, then $g=e$, 
\end{itemize}
for $i=1,2$. 
Then, there are projections
$(p_{v,w})_{v\in S, w\in T}$ and $(q_{w,v})_{v\in S, w\in T}$ such that
the following holds for every $v\in S$ and $w\in T$. 
\begin{enumerate}
\item $p_{v,w}\in \mathcal{Z}(B'_{\mathrm{Lk}v}\cap pPp)$, 
$q_{w,v}\in \mathcal{Z}(C'_{\mathrm{Lk}w}\cap qPq)$. \\
\item $\sum_{w'\in T} p_{v,w'}=p$, $\sum_{v'\in S} q_{w,v'}=q$  \\
\item $B_{\mathrm{Lk}v}p_{v,w} \prec^s_P C_{\mathrm{Lk}w}q_{w,v}$, 
$C_{\mathrm{Lk}w}q_{w,v} \prec^s_P B_{\mathrm{Lk}v}p_{v,w}$
\end{enumerate}
\end{theorem}

\begin{proof}
For each pair $(v,w)\in S\times T$, 
let $p_{v,w}\in \mathcal{Z}(B'_{\mathrm{Lk}v}\cap pPp)$ and $q_{w,v}\in \mathcal{Z}(C'_{\mathrm{Lk}w}\cap qPq)$ be the maximal projection 
such that $B_{\mathrm{Lk}v}p_{v,w}\prec^s C_{\mathrm{Lk}w}$ and 
$C_{\mathrm{Lk}w}q_{w,v}\prec^s B_{\mathrm{Lk}v}$. 
(These projections can be zero. )

We claim that $\bigvee_{w\in T} p_{v,w}=p$ and $\bigvee_{v\in S} q_{w,v}=q$
for each $v$ and $w$. 
Note that in this case $B$ has no amenable direct summand 
since it is a group von Neumann algebra
of a nonamenable group.

Take an arbitrary vertex $v\in S$ 
and a projection $r\in \mathcal{Z}(B'_{\mathrm{Lk}v}\cap pPp)$. 
Then, by Lemma 2.14 (5), $r$ commutes with $B_{\mathrm{St}v}$ and the relative commutant of $B_{\mathrm{Lk}v}r$ contains the nonamenable subalgebra $B_vr$. 
By the relative bi-exactness in Theorem B, 
there exists a vertex $w\in V\Gamma_2$ such that $B_{\mathrm{Lk}v}r\prec C_{\mathrm{Lk}w}$. 
By taking a conjugation by the canonical unitaries if necessary, one may take $w\in T$ 
and the claim is proven.  


For the rest part of the proof, 
just note that, as in the proof of Theorem 5.4, 
it is easily seen that 
$B_{\mathrm{Lk}v}p_{v,w}\prec^s C_{\mathrm{Lk}w}q_{w,v}$
and
$p_{v,w}p_{v,w'}= 0$ for any $v\in S$ and $w\neq w'\in T$. 
\end{proof}

We end this paper with a few comments on the rigidity problem for graph products and graph-wreath products. 

\begin{remark}
(1) By using Theorem A instead of Theorem B in the proof of Theorem 5.5 and then 
following the proof of \cite[Theorem 6.18]{bcc}, one can fully recover 
the information of the underlying graph for graph products on rigid graphs 
as in \cite[Theorem 6.18]{bcc} 
when the vertex groups are bi-exact and i.c.c.. 

(2) As mentioned at the end of Section 1, the author does not know if one can recover 
the underlying graph $\Gamma$ or the group $G$ themselves from the resulting von Neumann algebra. One obstruction to these global rigidity results is the difficulty in establishing the relation of the form $B_E \prec C_F$ where $E\nsubseteq \mathrm{St}v$ for any $v\in V\Gamma_1$. For example, if one gets $B_{\Gamma_1}\prec C_{\Gamma_2}$, then 
one easily obtain the isomorphism between $G_1$ and $G_2$ up to finite normal subgroups 
by \cite[Lemma 7.6]{im}. 

\end{remark}

\bibliographystyle{amsalpha}
\bibliography{MThesisRef}


\end{document}